\documentclass[11pt]{IEEEtran}
\usepackage{dsfont}

\usepackage{hyperref}
\usepackage{mathrsfs}

\usepackage[usenames]{color}

\usepackage[color,final]{showkeys}

\usepackage{amsmath,amssymb,amsthm}
\usepackage{latexsym,amsfonts,amscd,amsxtra,amstext}
\usepackage{algorithm,algorithmic}
\usepackage{url,cite}
\usepackage{hyperref}

\usepackage{epsfig} 
\usepackage{threeparttable}
\usepackage{subfigure}

\usepackage{mathtools}
\usepackage[normalem]{ulem} 
\usepackage{enumerate}
\usepackage{amssymb}
\usepackage{booktabs}
\usepackage{multirow,color}
\usepackage{amsfonts,dsfont}
\usepackage{hyperref}
\usepackage{cite}
\usepackage{pifont}
\usepackage{cases}

\usepackage[color]{showkeys}
\definecolor{refkey}{gray}{0.8}
\definecolor{labelkey}{gray}{0.8}

\newtheorem{definition}{Definition}

\newcommand{\remove}[1]{{}}
\newcommand{\cut}[1]{}

\newcommand{\ba}{\left[ \begin{array}}
	\newcommand{\ea}{\\ \end{array} \right]}
\newcommand{\qd}{\hfill{$\blacksquare$}}
\newcommand{\define}{\;\stackrel{\Delta}{=}\;}

\newcommand{\eq}[1]{\begin{align}#1\end{align}}
\def\tran{^{\mathsf{T}}}

\def\qed{\mbox{\rule[0pt]{1.3ex}{1.3ex}}}

\newcommand{\sw}{{\scriptstyle{\mathcal{W}}}}

\newcommand{\sx}{{\scriptstyle{\mathcal{X}}}}
\newcommand{\sy}{{\scriptstyle{\mathcal{Y}}}}

\newcommand{\sz}{{\scriptstyle{\mathcal{Z}}}}

\newcommand{\twd}{\widetilde{\scriptstyle{\mathcal{W}}}}

\newcommand{\tcA}{\overline{{\mathcal{A}}}}
\newcommand{\tA}{\overline{A}}

\newcommand{\nnb}{\nonumber \\}

\newcommand{\cA}{{\mathcal{A}}}

\newcommand{\cC}{{\mathcal{C}}}
\newcommand{\cD}{{\mathcal{D}}}

\newcommand{\cF}{{\mathcal{F}}}
\newcommand{\cG}{{\mathcal{G}}}
\newcommand{\cH}{{\mathcal{H}}}
\newcommand{\cI}{{\mathcal{I}}}
\newcommand{\cJ}{{\mathcal{J}}}

\newcommand{\cL}{{\mathcal{L}}}
\newcommand{\cM}{{\mathcal{M}}}
\newcommand{\cN}{{\mathcal{N}}}

\newcommand{\cP}{{\mathcal{P}}}

\newcommand{\cU}{{\mathcal{U}}}
\newcommand{\cV}{{\mathcal{V}}}

\newcommand{\RR}{\mathbb{R}}

\newcommand{\diag}{{\mathrm{diag}}} 
\newcommand{\col}{{\mathrm{col}}} 
\newcommand{\grad}{{\nabla}}    

\DeclareMathOperator*{\argmin}{arg\,min}

\newcommand{\bc}{\begin{center}}
	\newcommand{\ec}{\end{center}}

\newcommand{\bdm}{\begin{displaymath}}
	\newcommand{\edm}{\end{displaymath}}

\newcommand{\beq}{\begin{equation}}
	\newcommand{\eeq}{\end{equation}}

\newcommand{\bfl}{\begin{flushleft}}
	\newcommand{\efl}{\end{flushleft}}

\newcommand{\bt}{\begin{tabbing}}
	\newcommand{\et}{\end{tabbing}}

\newcommand{\beqn}{\begin{align}}
	\newcommand{\eeqn}{\end{align}}

\newcommand{\beqs}{\begin{align*}} 
	\newcommand{\eeqs}{\end{align*}}  

\newcommand{\HRule}{\rule{\linewidth}{0.5mm}}

\newtheorem{lemma}{{Lemma}}

\newtheorem{corollary}{{Corollary}}

\def\tran{^{\mathsf{T}}}

\def\real{{\mathbb{R}}}

\def\Zint{{\mathchoice{\setbox1=\hbox{\sf Z}\copy1\kern-.75\wd1\box1}
{\setbox1=\hbox{\sf Z}\copy1\kern-.75\wd1\box1}
{\setbox1=\hbox{\scriptsize\sf Z}\copy1\kern-.75\wd1\box1}
{\setbox1=\hbox{\scriptsize\sf Z}\copy1\kern-.75\wd1\box1}}}

\makeatletter
\def\hlinewd#1{%
  \noalign{\ifnum0=`}\fi\hrule \@height #1 \futurelet
   \reserved@a\@xhline}
\makeatother

\title{Exact Diffusion for Distributed Optimization and Learning --- Part I: Algorithm Development}
\author{\IEEEauthorblockN{Kun Yuan, Bicheng Ying, Xiaochuan Zhao,  and Ali H. Sayed \\}
	\vspace{0.4cm}
	
	\thanks{\scriptsize{K. Yuan and B. Ying are with the Department of Electrical Engineering, University of California, Los Angeles, CA 90095 USA. Email:\{kunyuan, ybc, xiaochuanzhao\}@ucla.edu. X. Zhao is now with Goldman Sachs, NY. A. H. Sayed is with the School of Engineering, Ecole Polytechnique Federale de Lausanne (EPFL), Switzerland. Email: ali.sayed@epfl.ch. This  work was supported in part by NSF grants CCF-1524250 and ECCS-1407712. {A short conference version of the results from Parts I and II appear in the short conference publication \cite{yuan2017eusipco}.}}}
}

\begin{document}
%
\maketitle
\small
\begin{abstract}
This work develops a distributed optimization strategy with guaranteed exact convergence for a broad class of left-stochastic combination policies. The resulting exact diffusion strategy is shown in Part II \cite{yuan2017exact2} to have a wider stability range and superior convergence performance than the  EXTRA strategy. The exact diffusion solution is applicable to non-symmetric  left-stochastic combination matrices, while many earlier developments on exact consensus implementations are limited to doubly-stochastic matrices; these latter matrices impose stringent constraints on the network topology. The derivation of the exact diffusion strategy in this work relies on reformulating the aggregate optimization problem as a penalized problem and resorting to a diagonally-weighted incremental construction. Detailed stability and convergence analyses are pursued in Part II \cite{yuan2017exact2} and are facilitated by examining the evolution of the error dynamics in a transformed domain. Numerical simulations illustrate the theoretical conclusions.
\end{abstract}

\section{Introduction and Motivation}
\setlength\abovedisplayskip{3pt}
\setlength\belowdisplayskip{3pt}
\setlength\abovedisplayshortskip{3pt}
\setlength\belowdisplayshortskip{3pt}


This work deals with {\em deterministic} optimization problems where a collection of $N$ networked agents operate cooperatively to solve an  aggregate optimization problem of the form:
\eq{
	\label{prob-consensus}
w^o = \argmin_{w\in \RR^M}\quad \cJ^o(w)=\sum_{k=1}^{N} J_k(w).
}
In this formulation, each risk function $J_k(w)$ is convex and differentiable, while the aggregate cost $\cJ^o(w)$ is strongly-convex. {\color{black}Throughout the paper, we assume the network is undirected.} All agents seek to determine the unique global minimizer, $w^o$, under the constraint that agents can only communicate with their neighbors. This distributed approach is robust to failure of links and/or agents and scalable to the network size. Optimization problems of this type find applications in a wide range of areas including  wireless sensor networks \cite{estrin2001instrumenting,rossi2004distributed,li2002detection,akyildiz2002survey}, multi-vehicle and multi-robot control systems \cite{ren2007information,zhou2011multirobot}, cyber-physical systems and smart grid implementations \cite{amin2005toward,ibars2010distributed,kim2011cloud,giannakis2013monitoring}, distributed adaptation and estimation \cite{sayed2014adaptive,sayed2014adaptation,chen2012diffusion,chen2015learning,chen2015learning2}, distributed statistical learning \cite{duchi2012dual,chen2015dictionary,chouvardas2012sparsity} and clustering \cite{zhao2015distributed,chen2015diffusion}.

There are several classes of distributed algorithms that can be used to solve problem \eqref{prob-consensus}. In the primal domain, implementations that are based on gradient-descent methods  are effective and easy to implement. There are at least two prominent variants under this class: the consensus strategy \cite{nedic2009distributed,dimakis2010gossip,kar2011convergence,kar2012distributed,yuan2016convergence, olfati2005consensus,sardellitti2010fast,braca2008running} and the diffusion strategy \cite{sayed2014adaptive,sayed2014adaptation,chen2012diffusion,chen2015learning,chen2015learning2}. There is a subtle but critical difference in the order in which computations are performed under these two strategies. In the consensus implementation, each agent runs a gradient-descent type iteration, albeit one where the starting point for the recursion and the point at which the gradient is approximated are {\em not} identical. This construction introduces an asymmetry into the update relation, which has some undesirable instability consequences {\color{black}(described, for example, in Secs. 7.2--7.3, Example 8.4, and also in Theorem 9.3 of \cite{sayed2014adaptation} and Sec. V.B and Example 20 of \cite{sayed2014adaptive})}. The diffusion strategy, in comparison, employs a symmetric update where the starting point for the iteration and the point at which the gradient is approximated coincide. This property results in a wider stability range for diffusion strategies \cite{sayed2014adaptive,sayed2014adaptation}. Still, when sufficiently small step-sizes are employed to drive the optimization process, both types of strategies (consensus and diffusion) are able to converge exponentially fast, albeit only to {\em an approximate solution} \cite{sayed2014adaptation,yuan2016convergence}. Specifically, it is proved in \cite{sayed2014adaptation,chen2013distributed,yuan2016convergence} that both the consensus and diffusion iterates under constant step-size learning converge towards a neighborhood of square-error size $O(\mu^2)$ around the true optimizer, $w^o$, i.e., $\|\widetilde{w}_{k,i}\|^2=O(\mu^2)$ as $i\rightarrow\infty$, where $\mu$ denotes the step-size and $\widetilde{w}_{k,i}$ denotes the error at agent $k$ and iteration $i$ relative to $w^o$. Since we are dealing with {\em deterministic} optimization problems, this small limiting bias is not due to any gradient noise arising from stochastic approximations; it is instead due to the inherent structure of the consensus and diffusion updates as clarified in the sequel.

Another important family of distributed algorithms are those based on the distributed alternating direction method of multipliers (ADMM) \cite{mateos2010distributed,mota2013d,shi2014linear} and its variants \cite{ling2015dlm,chang2015multi,mokhtari2015dqm}. These methods treat problem \eqref{prob-consensus} in both the primal and dual domains. It is shown in \cite{shi2014linear} that distributed ADMM with constant parameters will converge exponentially fast to the exact global solution $w^o$. However, distributed ADMM
solutions are computationally more expensive since they necessitate the solution of optimal sub-problems at each iteration. Some useful variations of distributed ADMM \cite{ling2015dlm,chang2015multi,mokhtari2015dqm} may alleviate the computational burden, but their recursions are still more difficult to implement than consensus or diffusion.

In more recent work \cite{shi2015extra}, a modified implementation of consensus iterations, referred to as EXTRA, is proposed and shown to converge to the {\em exact} minimizer $w^o$ rather than to an $O(\mu^2)-$neighborhood around $w^o$. The modification has a similar computational burden as traditional consensus and is based on adding a step that combines two prior iterates to remove bias. Motivated by \cite{shi2015extra}, other variations with similar properties were proposed in \cite{lorenzo2016next, nedich2016achieving, qu2017harnessing, xu2015augmented, nedic2016geometrically}. These variations rely instead on combining inexact gradient evaluations with a gradient tracking technique. The resulting algorithms, compared to EXTRA, have two information combinations per recursion, {\color{black}which doubles the amount of communication variables compared to EXTRA,} and can become a burden when communication resources are limited.



The current work is motivated by the following considerations. The result in \cite{shi2015extra} shows that the EXTRA technique resolves the bias problem in consensus implementations. However, it is known that traditional diffusion strategies outperform traditional consensus strategies. Would it be possible then to correct the bias in the diffusion implementation and attain an algorithm that is superior to EXTRA (e.g., an implementation that is more stable than EXTRA)? This is one of the contributions in this two-part work; Parts I and II\cite{yuan2017exact2}. In this part, we shall indeed develop a bias-free diffusion strategy that will be shown in Part II\cite{yuan2017exact2} to have a wider stability range than EXTRA consensus implementations. Achieving these objectives is challenging for several reasons. First, we need to understand the origin of the bias in diffusion implementations. Compared to the consensus strategy, the source of this bias is different and still not well understood. In seeking an answer to this  question, we will initially observe that the diffusion recursion can be framed as {an incremental algorithm} to solve a penalized version of \eqref{prob-consensus} and not \eqref{prob-consensus} directly --- see expression \eqref{prob-penalty-2} further ahead. In other words, the local diffusion estimate $w_{k,i}$, held by agent $k$ at iteration $i$, will be shown to approach the solution of a penalized problem rather than $w^o$, which causes the bias.


\subsection{{\color{black}Contributions}}
We have {three} main contributions in this article {and the accompanying Part II \cite{yuan2017exact2}} relating to: (a) developing a distributed algorithm that ensures exact convergence based on the diffusion strategy; (b) developing a strategy with wider stability range and enhanced performance than EXTRA consensus; and (c) developing a strategy with these properties for the larger class of local balanced (rather than only doubly-stochastic) matrices. 

To begin with, we will show in this article how to modify the diffusion strategy such that it solves the real problem \eqref{prob-consensus} directly. We shall refer to this variant as {\em exact diffusion}. Interestingly, the structure of {\em exact diffusion} will turn out to be very close to the structure of {\em standard} diffusion. The only difference is that there will be an extra ``correction'' step added between the usual ``adaptation'' and ``combination'' steps of diffusion --- see the listing of Algorithm 1 further ahead. It will become clear that this adapt-correct-combine (ACC) structure of the exact diffusion algorithm is {\color{black}more symmetric} in comparison to the EXTRA recursions. In addition, the computational cost of the ``correction'' step is trivial. Therefore, with essentially the same computational efficiency as standard diffusion, the exact diffusion algorithm will be able to converge {\em exponentially fast} to $w^o$ without any bias. Secondly, we will show in Part II\cite{yuan2017exact2} that exact diffusion has a wider stability range than EXTRA. In other words, there will exist a larger range of step-sizes that keeps exact diffusion stable but not the EXTRA algorithm. This is an important observation because larger values for $\mu$ help accelerate convergence.

Our third contribution is that we will derive the exact diffusion algorithm, and establish these desirable properties {\color{black}for the class of {\em locally balanced} combination matrices. This class does not only include symmetric doubly-stochastic matrices as special cases, but it also includes a range of widely-used left-stochastic policies as explained further ahead. First, we recall that  left-stochastic matrices are defined as follows.} Let $a_{\ell k}$ denote the weight that is used to scale the data that flows from agent $\ell$ to $k$. Let $A\define [a_{\ell k}]\in \RR^{N\times N}$ denote the matrix that collects all these coefficients. The entries on each column of $A$ are assumed to add up to one so that $A$ is {\em left-stochastic}, i.e., it holds that
\eq{\label{weight}
A\tran \mathds{1}_N = \mathds{1}_N,\quad \mbox{ or }\quad \sum_{\ell=1}^{N} a_{\ell k} = 1,\ \forall\, k=1,\cdots, N.
}
The matrix $A$ will not be required to be symmetric. For example, it may happen that $a_{\ell k}\neq a_{k \ell}$. 
Using these coefficients, when an agent $k$ combines the iterates $ \{\psi_{\ell,i}\}$ it receives from its
neighbors, that combination will correspond to a calculation of the form:
\eq{\label{weighted-average}
w_{k,i+1} = \sum_{\ell=1}^{N} a_{\ell k} \psi_{\ell, i}, \quad \mbox{where} \quad \sum_{\ell=1}^{N} a_{\ell k} = 1.
}

It should be emphasized that condition \eqref{weight}, which is repeated in \eqref{weighted-average}, is different from all previous algorithms studied in \cite{nedic2009distributed,mateos2010distributed,mota2013d,shi2014linear,chang2015multi,shi2015extra,xu2015augmented,nedic2016geometrically}, which require $A$ to be symmetric and doubly stochastic (i.e., each of its columns and rows should add up to one). {\color{black}Although symmetric doubly-stochastic matrices are common in distributed optimization, policies of great practical value happen to be left-stochastic and not doubly-stochastic. For example, it is shown in Chapters 12 and 15 of \cite{sayed2014adaptation} that the Hastings rule (see \eqref{a-lk}) and the relative-degree rule (see \eqref{xch}) achieve better mean-square-error (MSE) performance over adaptive networks than doubly-stochastic policies. Both of these rules are left-stochastic. Also, as we explain in Sec. \ref{sec-average-vs-ds}, the averaging rule (see \eqref{a-lk-ave}) leads to faster convergence in highly unbalanced networks where the degrees of neighboring nodes differ drastically. This rule is again left-stochastic and is rather common in applications involving data analysis over social networks. Furthermore, the averaging rule has better privacy-preserving properties than doubly-stochastic policies since it can be constructed from information available solely at the agent. In contrast, the doubly-stochastic matrices generated, for example, by the maximum-degree rule or Metropolis rule \cite{sayed2014adaptation} will require agents to share their degrees with neighbors.

{\color{black}We further remark that our proposed approach is different from existing algorithms that employ the useful push-sum technique, which requires $A$ to be right (rather than left) stochastic, i.e., $A$ is required to satisfy instead\footnote{\scriptsize Different from this paper, the notation $a_{\ell k}$ in \cite{tsianos2012push,nedic2016stochastic,xi2015linear,zeng2015extrapush,nedich2016achieving} is used to denote the weight that scales the data flowing from agent $k$ to $\ell$ (rather than from $\ell$ to $k$ as in this paper). From this notational viewpoint, the combination matrix $A$ in \cite{tsianos2012push,nedic2016stochastic,xi2015linear,zeng2015extrapush,nedich2016achieving} is left-stochastic rather than right-stochastic. 
	}$A\mathds{1}_N = \mathds{1}_N$. For instance, the push-sum implementations in \cite{tsianos2012push,nedic2016stochastic,xi2015linear, zeng2015extrapush,nedich2016achieving} replace the rightmost condition in \eqref{weighted-average} by
	\eq{\label{non-weighted-average}
		w_{k,i+1} = \sum_{\ell=1}^{N} a_{\ell k} \psi_{\ell, i}, \quad \mbox{where} \quad \sum_{k=1}^{N} a_{\ell k} = 1.
	}
	It will be illustrated in the simulations (later in Fig. 3 of Part II \cite{yuan2017exact2}) that the use of a left-stochastic combination policy and the adapt-then-combine structure in our approach lead to more efficient communications, and also to a stable performance over a wider range of step-sizes than right-stochastic policies used in the push-sum implementations \cite{tsianos2012push,nedic2016stochastic,xi2015linear, zeng2015extrapush,nedich2016achieving}.  However, the difference in the nature of the combination matrix (left vs. right-stochastic) complicates the convergence analysis and requires a completely different convergence analysis approach from \cite{tsianos2012push,nedic2016stochastic,xi2015linear, zeng2015extrapush,nedich2016achieving}. 
	
%
%
%
}

}

{In this Part I we derive the exact diffusion algorithm, while in Part II \cite{yuan2017exact2} we establish its convergence properties and prove its stability superiority over the EXTRA algorithm. This article is organized as follows. In Section \ref{sec-diffusion} we review the standard diffusion algorithm, introduce locally-balanced left-stochastic combination policies, and establish several of their properties. In Section \ref{sec-diffusion-penalty} we identify the source of bias in standard diffusion implementations. In Section \ref{sec-exact-diffusion} we design the exact diffusion algorithm to correct for the bias. In Section V we illustrate the necessity of the locally-balanced condition on the combination policies by showing that divergence can occur if it is not satisfied. Numerical simulations are presented in Section \ref{sec-simulation}.



{\em Notation:} Throughout the paper we use $\diag\{x_1,\cdots,x_N\}$ to denote a diagonal matrix consisting of diagonal entries ${x_1,\cdots,x_N}$,  and use $\col\{x_1,\cdots,x_N\}$ to denote a column vector formed by  stacking ${x_1,\cdots,x_N}$. For symmetric matrices $X$ and $Y$, the notation $X \le Y$ or $Y\ge X$ denotes $Y - X$ is positive semi-definite. For a vector $x$, the notation $x \succeq 0$ denotes that each element of $x$ is non-negative, while the notation $x \succ 0$ denotes that each element of $x$ is positive. For a matrix $X$, we let $\mathrm{range}(X)$ denote its range space, and $\mathrm{null}(X)$ denote its null space. The notation $\mathds{1}_N = \col\{1,\cdots,1\} \in \RR^{N}$.

\section{Diffusion and combination policies}
\label{sec-diffusion}

\subsection{Standard Diffusion Strategy}
To proceed, we will consider a more general optimization problem than \eqref{prob-consensus} by introducing a weighted aggregate cost of the form:
\eq{
	\label{prob-dist}
	w^\star = \argmin_{w\in \RR^M}\quad \cJ^\star(w) = \sum_{k=1}^{N} q_k J_k(w),
}
for some positive coefficients $\{q_k\}$. Problem \eqref{prob-consensus} is a special case when the $q_k$ are uniform, i.e., $q_1=q_2=\ldots=q_N$, in which case $w^{\star}=w^o$. Note also that the aggregate cost ${\cal J}^{\star}(w)$ is strongly-convex when ${\cal J}^{o}(w)$ is strongly-convex. 

To solve problem \eqref{prob-dist} over a {\em connected} network of agents, we consider the standard diffusion strategy \cite{chen2012diffusion,sayed2014adaptive,sayed2014adaptation}:
\eq{
	\psi_{k,i} &= w_{k,i-1} - \mu_k \grad J_k(w_{k,i-1}), \label{d-1}\\
	w_{k,i} &= \sum_{\ell \in \cN_k} a_{\ell k} \psi_{\ell,i}, \label{d-2}
}
where $\{\mu_k\}_{k=1}^N$ are positive step-sizes, and the $\{a_{\ell k}\}_{\ell=1,k=1}^N$ are nonnegative combination weights satisfying
\eq{\label{left-stochastic a-lk}
	\sum_{\ell \in \cN_k} a_{\ell k}=1.
}
Moreover, ${\cN}_k$ denotes the set of neighbors of agent $k$, and $\nabla J_k(\cdot)$ denotes the gradient vector of $J_k$ relative to $w$. It follows from \eqref{left-stochastic a-lk} that $A=[a_{\ell k}]\in\real^{N\times N}$ is a left-stochastic matrix. 
It is assumed that the network graph is connected, meaning that a path with nonzero combination weights can be found linking any pair of agents. It is further assumed that the graph is strongly-connected, which means that at least one diagonal entry of $A$ is non-zero \cite{sayed2014adaptation} (this is a reasonable assumption since it simply requires that at least one agent in the network has some confidence level in its own data). In this case, the matrix $A$ will be primitive. This implies, in view of the Perron-Frobenius theorem \cite{pillai2005perron,sayed2014adaptation}, that there exists an eigenvector $p$ satisfying 
\eq{
	Ap=p,\;\;\;\mathds{1}_N\tran p=1,\;\;p \succ 0.
}
We refer to $p$ as the Perron eigenvector of $A$. Next, we introduce the vector
\eq{
	q\define \mbox{\rm col}\{q_1,q_2,\ldots,q_N\} \in \RR^N,
}
where $q_k$ is the weight associated with $J_k(w)$ in \eqref{prob-dist}. Let {the constant scalar} $\beta$ be chosen such that
\eq{\label{q-A}
	q = \beta \, \mbox{diag}\{\mu_1,\mu_2,\cdots,\mu_N\}p.
}
\vspace{1mm}
\noindent {\bf \hspace{-1.2mm}Remark 1. (Scaling)} Condition \eqref{q-A} is not restrictive and can be satisfied for any left-stochastic matrix $A$ through the choice of the parameter $\beta$ and the step-sizes. Note that $\beta$ should satisfy 
\eq{ \beta =\frac{q_k}{p_k}\frac{1}{\mu_k}}
for all $k$. To make the expression for $\beta$ independent of $k$, we parameterize (select) the step-sizes as 
\eq{\label{step_size_form} \mu_k=\left(\frac{q_k}{p_k}\right) \mu_o}
for some small $\mu_o > 0$. Then, $\beta = 1/\mu_o$, which is independent of $k$, and relation (11) is satisfied. 

\rightline \qed

{\color{black}
\noindent {\bf Rermark 2. (Perron entries)} Expression \eqref{step_size_form} suggests that agent $k$ needs to know the Perron entry $p_k$ in order to run the diffusion strategy \eqref{d-1}--\eqref{d-2}. As we are going to see in the next section, the Perron entries are actually available beforehand and in closed-form for several well-known left-stochastic policies (see, e.g., expressions \eqref{p_k}, \eqref{p_k-ave}, and \eqref{xcnj} further ahead). For other left-stochastic policies for which closed-form expressions for the Perron entries may not be available, these can be determined iteratively by means of the power iteration --- see, e.g., the explanation leading to future expression \eqref{power_iteration_result}. \qd
}

It was shown by Theorem 3 in \cite{chen2013distributed} that under \eqref{q-A}, the iterates $w_{k,i}$ generated through the diffusion recursion \eqref{d-1}-\eqref{d-2} will approach $w^\star$, i.e., 
\eq{\label{diffusion-limsup}
\limsup_{i\to \infty} \|w^\star-w_{k,i}\|^2 = O(\mu_{\max}^2), \ \forall\; k=1,\cdots, N,
}
where $\mu_{\max} = \max\{\mu_1,\cdots, \mu_N\}$. Result \eqref{diffusion-limsup} implies that the diffusion algorithm will converge to a neighborhood around $w^\star$, and that the square-error bias is on the order of $O(\mu_{\max}^2)$.

\subsection{Combination Policy}\label{subsec-A}
Result \eqref{diffusion-limsup} is a reassuring conclusion: it ensures that the squared-error is small whenever $\mu_{\max}$ is small; moreover, the result holds for {\em any} left-stochastic matrix. Moving forward, we will focus on an important subclass of left-stochastic matrices, namely, those that satisfy a mild {\em local balance} condition (we shall refer to these matrices as {\em balanced} left-stochastic policies)\cite{zhao2014learning}. The balancing condition turns out to have a useful physical interpretation and, in addition, it will be shown to be satisfied by several widely used left-stochastic combination policies. The local balance condition will help endow networks  with crucial properties to ensure exact convergence to $w^\star$ without any bias. In this way, we will be able to propose distributed optimization strategies with exact convergence guarantees for this class of left-stochastic matrices, while earlier exact convergence results are limited to (the less practical) right-stochastic or doubly-stochastic policies; these choices face implementation difficulties for the reasons explained before, which is the main motivation for focusing on left-stochastic policies in our treatment. 
%
%
\begin{definition}[\sc Locally balanced Policies]\label{ass-lb}
	Let $p$ denote the Perron eigenvector of a primitive left-stochastic matrix $A$, with entries $\{p_{\ell}\}$. Let $P=\mbox{\rm diag}(p)$ correspond to the diagonal matrix constructed from $p$. The matrix $A$ is said to satisfy a local balance condition if it holds that
	\eq{\label{local-balance}
		a_{\ell k}\, p_k = a_{k \ell}\,p_{\ell}, \quad k,\ell =1,\cdots,N
	}
	or, equivalently, in matrix form:
	\eq{\label{lb-compact}
		PA\tran = AP.
	}
	\rightline \qed
\end{definition}

\noindent Relations of the form \eqref{local-balance} are common in the context of Markov chains. They are used there to model an equilibrium scenario for the probability flux into the Markov states \cite{whittle1968equilibrium,norris1998markov}, where the $\{a_{\ell k}\}$ represent the transition probabilities from states $\ell$ to $k$ and the $\{p_{\ell}\}$ denote the steady-state distribution for the Markov chain. 

We provide here an interpretation for \eqref{local-balance} in the context of multi-agent networks by considering two generic agents, $k$ and $\ell$, from an arbitrary network, as shown in Fig. \ref{fig:local-balance}. The coefficient $a_{\ell k}$ is used by agent $k$ to scale information arriving from agent $\ell$. Therefore, this coefficient reflects the amount of confidence that agent $k$ has in the information arriving from agent $\ell$. Likewise, for $a_{k\ell}$. Since the combination policy is not necessarily symmetric, it will hold in general that $a_{\ell k}\neq a_{k\ell}$. However, agent $k$ can re-scale the incoming weight $a_{\ell k}$ by $p_k$, and likewise for agent $\ell$, so that the local balance condition \eqref{local-balance} requires each pair of rescaled weights to match each other. We can interpret $a_{\ell k}$ to represent the (fractional) amount of information flowing from $\ell$ to $k$ and $p_k$ to represent the price paid by agent $k$ for that information. Expression \eqref{local-balance} is then requiring the information-cost benefit to be equitable across agents.
\begin{figure}[h!]
	\centering
	\includegraphics[scale=0.4]{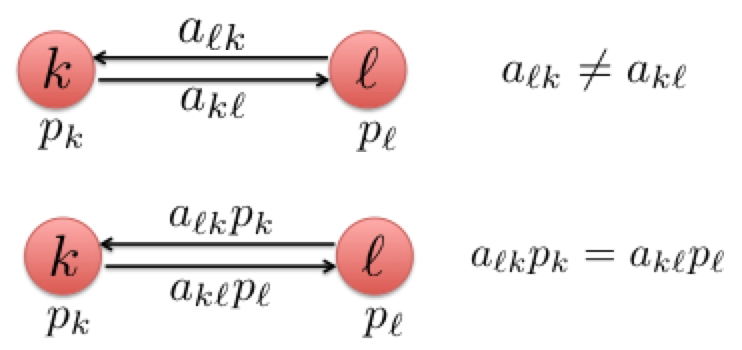}
	\caption{Illustration of the local balance condition \eqref{local-balance}.}
	\label{fig:local-balance}
\end{figure} 

It is worth noting that the local balancing condition \eqref{local-balance} is satisfied by several important left-stochastic policies, as illustrated in four examples below. Thus, let 
$\tau_k=\mu_{k}/\mu_{\max}$ for agent $k$. Then condition \eqref{q-A} becomes
\eq{\label{q-A-2}
	q = \beta \mu_{\max}\, \mbox{diag}\{\tau_1,\tau_2,\cdots,\tau_N\} p,
}
where $\tau_k\in (0, 1]$.

%
%
%
%
%
%

\vspace{1mm}
\noindent \textbf{Policy 1 (Hastings rule)} The first policy we consider is the Hastings rule. Given $\{q_k\}_{k=1}^N$ and $\{\mu_k\}_{k=1}^N$, we select $a_{\ell k}$ as \cite{sayed2014adaptation,hastings1970monte}:
\eq{\label{a-lk}
\hspace{-3mm}	a_{\ell k}=
	\begin{cases}
		\begin{array}{ll}\displaystyle
			\hspace{-2mm}\frac{\mu_k/q_k}{\max\{n_k\mu_k/q_k, n_\ell \mu_\ell /q_\ell \}},& \mbox{if $\ell \in \cN_k \backslash \{k\}$}, \vspace{2mm}\\
			\hspace{-2mm}\displaystyle 1 - \sum_{m\in \cN_k \backslash \{k\}}a_{mk}, & \mbox{if $\ell = k$},\\
			\hspace{-2mm}0,& \mbox{if $\ell \notin \cN_k$}.
		\end{array}
	\end{cases}
}
where $n_k \define |\cN_k|$ (the number of neighbors of agent $k$). It can be verified that $A$ is left-stochastic, and that the entries of its Perron eigenvector $p$ are given by
\eq{\label{p_k}
	p_k \define \frac{q_k/\mu_k}{\sum_{\ell=1}^{N}q_{\ell}/\mu_{\ell}}>0.
}
Let 
\eq{\label{beta-hastings}
	\beta=\sum_{\ell=1}^{N}q_{\ell}/\mu_{\ell} = \frac{1}{\mu_{\max}}\sum_{\ell=1}^{N}q_{\ell}/\tau_{\ell} > 0.
} 
With \eqref{a-lk} and \eqref{p_k}, it is easy to verify that 
\eq{
	a_{\ell k}p_k = \frac{1}{\beta \max\{n_k\mu_k/q_k, n_\ell \mu_\ell /q_\ell \}} = a_{k \ell}p_{\ell}.
}
If $\ell = k$, it is obvious that \eqref{local-balance} holds. If $\ell \notin \cN_k$, then $k \notin \cN_\ell$. In this case, $a_{\ell k}p_k = a_{k \ell}p_{\ell} = 0$. 

Furthermore, we can also verify that when $\{q_k\}_{k=1}^N$ and $\{\mu_k\}_{k=1}^N$ are given, $\{a_{\ell k}\}$ are generated through \eqref{a-lk}, and $\beta$ is chosen as in \eqref{beta-hastings}, then condition \eqref{q-A} is satisfied.

\rightline \qed

\vspace{2mm}
\noindent \textbf{Policy 2 (Averaging rule)} The second policy we consider is the popular average combination rule where $a_{\ell k}$ is chosen as
\eq{
	\label{a-lk-ave}
	a_{\ell k}=
	\begin{cases}
		\begin{array}{ll}
			\displaystyle 1/n_k, & \mbox{if $\ell \in \cN_k$},\\
			\displaystyle 0, & \mbox{otherwise}.
		\end{array}
	\end{cases}
}
The entries of the Perron eigenvector $p$ are given by
\eq{
	\label{p_k-ave}
	p_k = n_k\left(\sum_{m=1}^{N}n_m\right)^{-1}.
}
With \eqref{a-lk-ave} and \eqref{p_k-ave}, it clearly holds that
\eq{
	a_{\ell k}p_k = \left(\sum_{m=1}^{N}n_m\right)^{-1}=a_{k \ell}p_\ell,
}
which implies \eqref{local-balance}. 

We can further verify that when $\mu_k$ is set as 
%
\eq{\label{mu-ave}
	\mu_k = \frac{q_k}{n_k}\mu_o,\quad \forall\, k=1,2,\cdots,N
}
for some positive constant step-size $\mu_o$ and $\beta$ is set as
\eq{\label{beta-averaging}
	\beta = \left(\sum_{m=1}^{N}n_m\right) \Big/\mu_o > 0,
} 
then condition \eqref{q-A} will hold. \hspace{4.3cm}\qed

{\color{black}
	\vspace{2mm}
	\noindent \textbf{Policy 3 (Relative-degree rule)} The third policy we consider is the relative-degree combination rule \cite{cattivelli2010diffusion} where $a_{\ell k}$ is chosen as
	\eq{\label{xch}
		a_{\ell k} = 
		\begin{cases}
			n_{\ell} \left( \sum_{m\in \cN_k} n_m \right)^{-1}, & \mbox{if $\ell \in \cN_k$},\\
			0, & \mbox{otherwise},
		\end{cases}
	}
	and the entries of the Perron eigenvector $p$ are given by
	\eq{\label{xcnj}
		p_k = \frac{n_k \sum_{m\in \cN_k} n_m }{\sum_{k=1}^{N} \left(n_k \sum_{m\in \cN_k} n_m\right)}.
	}
	With \eqref{xch} and \eqref{xcnj}, it clearly holds that
	\eq{
		a_{\ell k}p_k = \frac{n_k n_{\ell}}{\sum_{k=1}^{N} \left(n_k \sum_{m\in \cN_k} n_m\right)}=a_{k \ell}p_\ell,
	}
	which implies \eqref{local-balance}. 
	
	We can further verify that when $\mu_k$ is set as 
	\eq{\label{b876}
		\mu_k = \frac{q_k}{n_k \sum_{m\in \cN_k} n_m}\mu_o,\quad \forall\, k=1,2,\cdots,K,
	}
	and $\beta$ is set as
	\eq{\label{beta-averaging-4}
		\beta = \sum_{k=1}^{N} \left(n_k \sum_{m\in \cN_k} n_m\right) \Big/\mu_o,
	} 
	then condition \eqref{q-A} will hold. \hspace{4.25cm}\qed
	
}

%
%
%
%
%

\vspace{2mm}
\noindent \textbf{Policy 4 (Doubly stochastic policy)} If matrix $A$ is primitive, symmetric, and doubly stochastic, its Perron eigenvector is $p=\frac{1}{N}\mathds{1}_N$.
In this situation, the local balance condition \eqref{local-balance} holds automatically.

 Furthermore, if we assume each agent employs the step-size $\mu_k = q_k N\mu_o$ for some positive constant step-size $\mu_o$, it can be verified that condition \eqref{q-A} holds with
\eq{\label{xchsss}
	\beta = {1}/{\mu_o}.
} 
There are various rules to generate a primitive, symmetric and doubly stochastic matrix. Some common rules are the Laplacian rule, maximum-degree rule, Metropolis rule and other rules that listed in Table 14.1 in \cite{sayed2014adaptation}. 

\rightline \qed

{\color{black}
\noindent \textbf{Policy 5 (Other locally-balanced policies)} For other left-stochastic-policies for which closed-form expressions for the Perron entries need not be available, the Perron eigenvector $p$ can be learned iteratively to ensure that the step-sizes $\mu_k$ end up satisfying \eqref{step_size_form}. Before we explain how this can be done, we remark that since the combination matrix $A$ is left-stochastic in our formulation, the power iteration employed in push-sum implementations cannot be applied since it works for right-stochastic policies. We proceed instead as follows.

Since $A$ is primitive and left-stochastic, it is shown in \cite{nedic2015distributed,sayed2014adaptation} that 
\eq{\label{23bsnsns}
	\lim_{i\to \infty} A^i = p \mathds{1}_N\tran.
}
This relation also implies 
\eq{\label{2bsns99}
	\lim_{i\to \infty} (A\tran)^i = \mathds{1}_N p\tran.
}
Now let $e_k$ be the $k$-th column of the identity matrix $I_N\in \RR^{N\times N}$. Furthermore, let each agent $k$ keep an auxiliary variable $z_{k,i}\in \RR^N$ with each $z_{k,-1}$ initialized to $e_k$. We also introduce 
\eq{
	\sz_i &\define \col\{z_{1,i},z_{2,i},\cdots, z_{N,i}\}\in \RR^{N^2}, \\
	\cA &\define A \otimes I_N.
}
By iterating $\sz_i$ according to 
\eq{\label{power_iteartion}
	\sz_{i+1} = \cA\tran \sz_i,
}
we have
\eq{\label{zn2bs00}
	\lim_{i\to \infty} \sz_i &= \lim_{i\to \infty} (\cA\tran)^{i+1}\sz_{-1}  \nnb
	&= \lim_{i\to \infty} [(A\tran)^{i+1} \otimes I_N] \sz_{-1} \overset{\eqref{2bsns99}}{=} (\mathds{1}_Np\tran \otimes I_N)\sz_{-1} \nnb
	&= [(\mathds{1}_N \otimes I_N)(p\tran \otimes I_N)]\sz_{-1}.
}
Since $\sz_{-1}=\col\{e_1,\cdots.e_N\}$, it can be verified that $(p\tran \otimes I_N)\sz_{-1} = p$.
Substituting into \eqref{zn2bs00}, we have
$\lim_{i\to \infty} z_{k,i} = p.$
In summary, it holds that
\eq{\label{power_iteration_result}
 \lim_{i\to \infty} z_{k,i}(k) = p_k，
}
where $z_{k,i}(k)$ is the $k$-th entry of the vector $z_{k,i}$. Therefore, if we set 
\eq{\label{xw3ndn}
	\mu_{k,i} = \frac{q_k \mu_o}{z_{k,i}(k)},
}
then it follows that 
\eq{
	\lim_{i\to \infty} \mu_{k,i} = {q_k \mu_o}/{p_k}.
}
{\color{black}
We finally note that the quantity $z_k(i)$ that appears in the denominator of \eqref{xw3ndn} can be guaranteed non-zero. This can be seen as follows.  From the power iteration \eqref{power_iteartion}, we have
\eq{\label{power-iteration-1}
	z_{k,i} & = \sum_{\ell \in \cN_k} a_{\ell k} z_{\ell, i-1}, \quad \forall\, k=1,\cdots,N.
}
Since $z_{k, -1} \succeq 0$ for any $k \in \{1,\cdots, N\}$ and the combination matrix $A$ has non-negative entries, we conclude that $z_{k,i} \succeq 0$ for $i \ge 0$. In addition, focusing on the $k$-th entry, we have
\eq{
	z_{k,i}(k) & = \sum_{\ell \in \cN_k} a_{\ell k} z_{\ell, i-1}(k) \nnb
	&= a_{kk}z_{k,i-1}(k) + \sum_{\ell \in \cN_k \setminus \{k\}} a_{\ell k} z_{\ell, i-1}(k) \nnb
	&\overset{(*)}{\ge} a_{kk}z_{k,i-1}(k) \ge (a_{kk})^{i+1} z_{k, -1}(k)
}
where the inequality (*) holds because $a_{\ell k}\ge 0$ and $z_{\ell, i-1}(k) \ge 0$. Since $z_{k,-1}(k)=1$, if we let $a_{kk} > 0$, i.e., each agent assigns positive weight to itself, we have
\eq{
	z_{k,i}(k) \ge (a_{kk})^{i+1} > 0,\ \forall k =1, \cdots, N, \ \forall i \ge 0.
}
In other words, the condition $a_{kk}>0$ can guarantee the positiveness of $z_{k,i}(k)$. This condition is not restrictive because, for example, we can replace the power iteration \eqref{power-iteration-1} by 
\eq{\label{bs8ks}
z_{k,i} = \sum_{\ell \in \cN_k} \bar{a}_{\ell k} z_{\ell, i-1}, \quad \forall\, k=1,\cdots,N,
}
where we are using the coefficients $\{\bar{a}_{\ell k}\}$ instead of $\{a_{\ell k}\}$. This is possible because the matrices $A$ and $\bar{A} = (I+A)/2$ are both left-stochastic and have the same Perron vector $p$. Note that $\bar{a}_{kk} = (a_{kk}+1)/2 > 0$ no matter whether $a_{kk}$ is zero or not.
}
}
\qd

\begin{figure}[h!]
	\centering
	\includegraphics[scale=0.28]{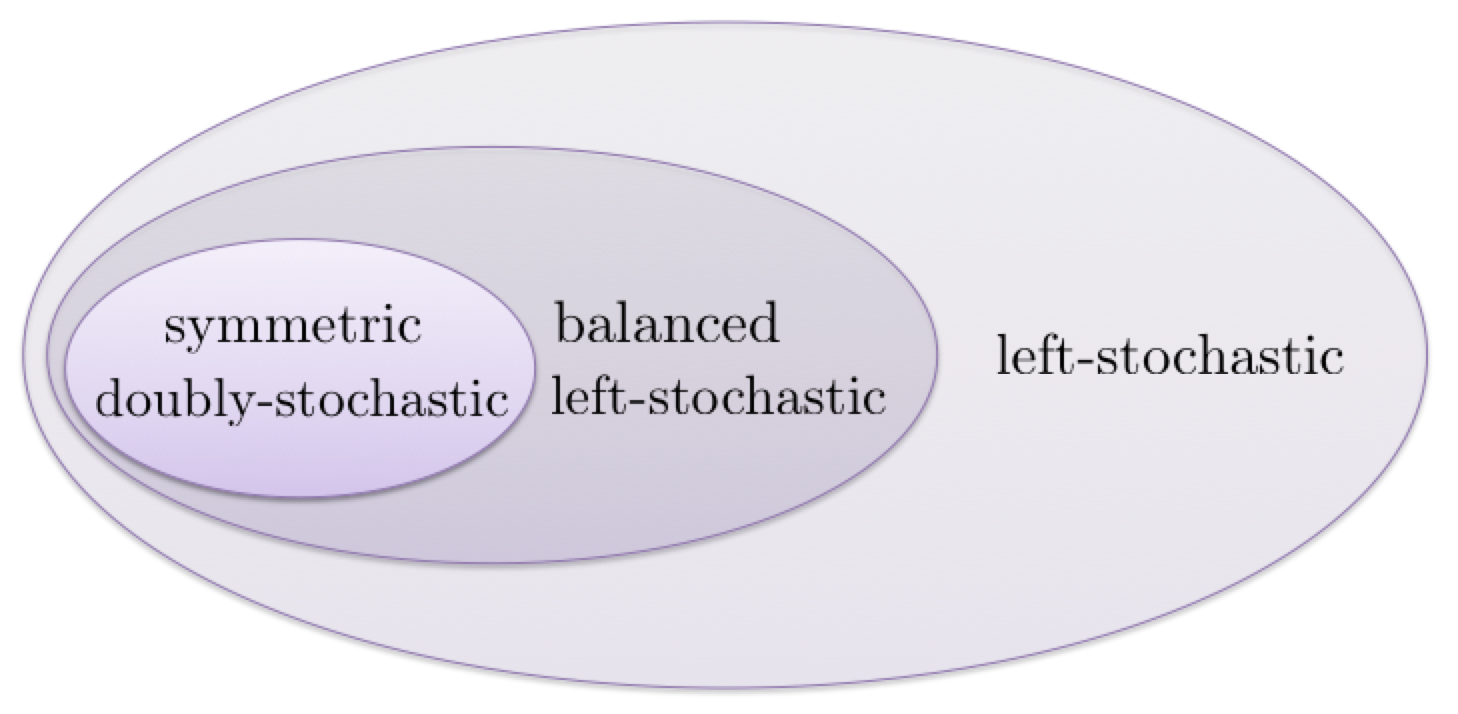}
	\caption{Illustration of the relations among the classes of symmetric doubly-stochastic, balanced left-stochastic, and left-stochastic combination matrices. }
	\label{fig:comb_policy_relations}
\end{figure}

We illustrate in Fig. \ref{fig:comb_policy_relations} the relations among the classes of symmetric doubly-stochastic, balanced left-stochastic, and left-stochastic combination matrices. It is seen that every symmetric doubly-stochastic matrix is both left-stochastic and balanced. We indicated earlier that the EXTRA consensus algorithm was derived in \cite{shi2015extra} with exact convergence properties for symmetric doubly-stochastic matrices. Here, in the sequel, we shall derive an exact diffusion strategy with exact convergence guarantees for the larger class of balanced left-stochastic matrices (which is therefore also applicable to symmetric doubly-stochastic matrices). We will show in Part II\cite{yuan2017exact2} that the exact diffusion implementation has a wider stability range than EXTRA consensus; this is a useful property since larger step-sizes can be used to attain larger convergence rates. 

\vspace{2mm}
\noindent {\color{black}\textbf{Remark 3. (Convergence guarantees)} One may wonder whether exact convergence can be guaranteed for the general left-stochastic matrices that are not necessarily balanced (i.e., whether the convergence property can be extended beyond the middle elliptical area in Fig. \ref{fig:comb_policy_relations}). It turns out that one can provide examples of combination matrices that are left-stochastic (but not necessarily balanced) for which exact convergence occurs and others for which exact convergence does not occur (see, e.g., the examples in Section \ref{sec-general-A} and  Figs. \ref{fig:general_A} and \ref{fig:general_A_converging}). In other words, exact convergence is not always guaranteed beyond the balanced class. This conclusion is another useful contribution of this work; it shows that there is a boundary inside the set of left-stochastic matrices within which convergence can be always guaranteed (namely, the set of balanced matrices).
	
It is worth noting that the recent works \cite{zeng2015extrapush,xi2015linear} extend the consensus-based EXTRA method to the case of directed networks by employing a push-sum technique. These extensions do not require the local balancing condition but they establish convergence only if the step-size parameter falls within an interval $(c_{\rm lower}, c_{\rm upper})$ where $c_{\rm lower}$ and $c_{\rm upper}$ are two positive constants. However, it is not proved in these works whether this interval is feasible, i.e., whether $c_{\rm upper} > c_{\rm lower}$. In fact, we will construct examples in Section \ref{sec-general-A} for which both exact diffusion and push-sum EXTRA will diverge for any step-size $\mu$. In other words, both exact diffusion and EXTRA methods need not work well for directed networks. This is a disadvantage in comparison with DIGing-based methods \cite{lorenzo2016next, nedich2016achieving, qu2017harnessing, xu2015augmented, nedic2016geometrically}.

In summary, when locally-balanced policies is employed, exact diffusion is more communication efficient and also more stable than other techniques including DIGing methods and EXTRA. However, just like EXTRA, the exact diffusion strategy is applicable to undirected (rather than directed) graphs. 
%
\rightline \qd
}


	
}




\subsection{Useful Properties}
We now establish several useful properties for primitive  left-stochastic matrices that satisfy the local balance condition \eqref{local-balance}. These properties will be used in the sequel.

\begin{lemma}[\sc Properties of $AP-P+I_N$]\label{prop-ds-matrix}
When $A$ satisfies the local balance condition \eqref{local-balance}, it holds that 
	the matrix ${AP - P} + I_N$ is primitive, symmetric, and doubly stochastic. 
	%
\end{lemma}
\begin{proof}
	With condition \eqref{local-balance}, the symmetry of $AP-P+I_N$ is obvious. To check the primitiveness of $AP-P+I_N$, we need to verify two facts, namely, that: (a) at least one diagonal entry in $AP-P+I_N$ is positive, and (b) there exists at least one path with nonzero weights between any two agents. It is easy to verify condition (a) because $A$ is already primitive and $P< I_N$. For condition (b), since $A$ is connected and all diagonal entries of $P$ are positive, then if there exists a path with nonzero coefficients linking agents $k$ and $\ell$ under $A$, the same path will continue to exist under $AP$. Moreover, since all diagonal entries of $ -P+I_N$ are positive, then the same path will also exist under $AP-P+I_N$. Finally, ${AP - P} + I_N$ is doubly stochastic because
	\eq{
		\mathds{1}_N\tran \left(AP - P + I_N \right) &= p\tran - p\tran + \mathds{1}_N\tran = \mathds{1}_N\tran, 
	}
	\eq{
		\left(AP - P + I_N \right)\mathds{1}_N &= p - p + \mathds{1}_N= \mathds{1}_N. \label{xcnh90}
	}
\end{proof}

\begin{lemma}[\sc Nullspace of $P-AP$]\label{coro-psd}
	When $A$ satisfies the local balance condition \eqref{local-balance}, it holds that $P-AP$ is symmetric and positive semi-definite. Moreover, it holds that 
	\eq{\label{null-P-AP}
		\mathrm{null}(P-AP) = \mathrm{span}\{\mathds{1}_N\},
	}
	where $\mathrm{null}(\cdot)$ denotes the null space of its matrix argument.
\end{lemma}
\begin{proof}
	Let $\lambda_k$ denote the $k$-th largest eigenvalue of ${AP - P} + I_N$. Recall from Lemma \ref{prop-ds-matrix} that $AP - P + I_N$ is primitive and doubly stochastic. Therefore, according to Lemma F.4 from \cite{sayed2014adaptation} it holds that
	\eq{\label{dzxcn}
		1 = \lambda_1 > \lambda_2 \ge \lambda_3 \ge \cdots \ge \lambda_N > -1,
	}
	It follows that the eigenvalues of $AP-P$ are non-positive so that $P - AP \ge 0$.
	
	Note further from \eqref{dzxcn} that the matrix ${AP - P}+ I_N$ has a single eigenvalue at one with multiplicity one. Moreover, from \eqref{xcnh90} we know that the vector $\mathds{1}_N$ is a right-eigenvector associated with this eigenvalue at one. Based on these two facts, we have
	\eq{\label{xzch-1}
	\left( {AP - P} + I_N \right)x =x \Longleftrightarrow x = c \mathds{1}_N
	}
	for any constant $c$. Relation \eqref{xzch-1} is equivalent to
	\eq{
	\left({AP - P}\right)x =0 \Longleftrightarrow x = c \mathds{1}_N,
	}
	which confirms \eqref{null-P-AP}. 
\end{proof}
\vspace{1mm}
\begin{corollary}[\sc Nullspace of $\cP - \cA \cP$]\label{coro-nullspace-P-AP}
	Let $\cP\define P\otimes I_M$ and $\cA \define A\otimes I_M$. When $A$ satisfies the local balance condition \eqref{local-balance}, it holds that
	\eq{\label{xcn987}
		\mathrm{null}( \cP - \cA \cP ) &=  \mathrm{null}\Big( (P - AP) \otimes I_M \Big) \nnb
		&= \mathrm{span}\{ \mathds{1}_N \otimes I_M \}.
	}
	Moreover, for any block vector $\sx=\col\{x_1,x_2,\cdots,x_N\} \in \RR^{MN}$ in the nullspace of ${\cal P}-{\cal A}{\cal P}$ with entries $x_k\in \RR^M$, it holds that
	\eq{\label{xcn86778}
		( \cP - \cA \cP )\sx=0 \Longleftrightarrow x_1 = x_2 =\cdots = x_N.
	}
\end{corollary}
\begin{proof}
	Since $P-AP+I_N$ has a single eigenvalue at $1$ with multiplicity one, we conclude that $(P-AP+I_N)\otimes I_M$ will have an eigenvalue at $1$ with multiplicity $M$. Next we denote the columns of the identity matrix by $I_M=[e_1, e_2, \cdots, e_N]$ where $e_k\in \RR^M$. We can verify that $\mathds{1}_N\otimes e_k$ is a right-eigenvector associated with the eigenvalue $1$ because
	\begin{align}
	&\hspace{-10mm} [(P-AP+I_N)\otimes I_M][\mathds{1}_N\otimes e_k]\nnb
	=&\ [(P-AP+I_N)\mathds{1}_N] \otimes e_k =  \mathds{1}_N \otimes e_k.
	\end{align}
	Now since any two vectors in the set $\{\mathds{1}_N \otimes e_k\}_{k=1}^M$ are mutually independent, we conclude that 
	\begin{align}
	 (\cP - \cA \cP) \sx = 0 \Longleftrightarrow &\ (\cP - \cA \cP + I_{MN}) \sx = \sx\nnb
	\Longleftrightarrow &\ \sx \in \mathrm{span}\{[\mathds{1}_N \otimes e_1,\cdots, \mathds{1}_N \otimes e_M]\}\nnb
	\Longleftrightarrow &\ \sx \in \mathrm{span}\{ \mathds{1}_N \otimes I_M \}.
	\end{align}	
	These equalities establish \eqref{xcn987}. From \eqref{xcn987} we can also conclude \eqref{xcn86778} because
	\eq{
		&\hspace{-15mm} \sx \in \mathrm{span}\{ \mathds{1}_N \otimes I_M \} \nnb
		\Rightarrow	&\ \sx = (\mathds{1}_N \otimes I_M) \cdot x = \col\{x,x,\cdots,x\}
	}
	from some $x\in \RR^M$. The direction ``$\Leftarrow$" of \eqref{xcn86778} is obvious.
\end{proof}

\begin{lemma}[\sc Real eigenvalues] \label{coro-real-eigenvalues}
	When $A$ satisfies the local balance condition \eqref{local-balance}, it holds that $A$ is diagonalizable with real eigenvalues in the interval $(-1,1]$, i.e.,
	\eq{\label{A-decomposition}
		A = Y \Lambda Y^{-1},
	}
	where $\Lambda = {\rm diag}\{\lambda_1(A), \cdots, \lambda_N(A)\} \in \RR^{N\times N}$, and 
	\eq{\label{A-eig-val}
		1 = \lambda_1(A) \hspace{-0.2mm} >  \hspace{-0.2mm} \lambda_2(A)  \hspace{-0.2mm}\ge \hspace{-0.2mm} \lambda_3(A)  \hspace{-0.2mm}\ge \hspace{-0.2mm} \cdots  \hspace{-0.2mm}\ge \hspace{-0.2mm} \lambda_N(A)  \hspace{-0.2mm} >  \hspace{-0.2mm} -1.
	}
\end{lemma}
\begin{proof}
	According to the local balance condition \eqref{lb-compact}, $PA\tran$ is symmetric. Using the fact that $P>0$ is diagonal, it holds that 
	\eq{
		P^{-\frac{1}{2}}A P^{\frac{1}{2}} = P^{-\frac{1}{2}} (AP) P^{-\frac{1}{2}},
	}
	which shows that the matrix on the left-hand side is symmetric. Therefore, $P^{-\frac{1}{2}}A P^{\frac{1}{2}}$ can be decomposed as 
	\eq{
		P^{-\frac{1}{2}}A P^{\frac{1}{2}} = Y_1 \Lambda Y_1\tran,\label{xcn76}
	}
	where $Y_1$ is an orthogonal matrix and $\Lambda$ is a real diagonal matrix. From \eqref{xcn76}, we further have that
	\eq{
		A = P^{\frac{1}{2}} Y_1 \Lambda Y_1\tran P^{-\frac{1}{2}}.
	}
	If we let $Y=P^{\frac{1}{2}} Y_1$, we reach the decomposition \eqref{A-decomposition}. Moreover, since $A$ is a primitive left-stochastic matrix, according to Lemma F.4 in \cite{sayed2014adaptation}, the eigenvalues of $A$ satisfy \eqref{A-eig-val}.
\end{proof}

For ease of reference, we collect in Table \ref{tab:table} the properties established in Lemmas \ref{prop-ds-matrix} through \ref{coro-real-eigenvalues} for balanced primitive left-stochastic matrices $A$.

\begin{table}[h]
	\centering
	\caption{}\vspace{-5mm}
	\label{tab:table}
	\begin{tabular}{l}
		\hspace{-2mm}\HRule\\
		\hspace{-2mm}\textbf{Properties of balanced primitive left-stochastic matrices $A$}\vspace{-2mm}\\
		\hspace{-2mm}\HRule\\
		\hspace{-2mm}$A$ is diagonalizable with real eigenvalues in $(-1,1]$;\\
		\hspace{-2mm}$A$ has a single eigenvalue at $1$;\\
		\hspace{-2mm}$AP-P+I_N$ is symmetric, primitive, doubly-stochastic;\\
		\hspace{-2mm}$P-AP$ is positive semi-definite;\\
		\hspace{-2mm}$\mathrm{null}(P-AP)=\mathrm{span}(\mathds{1}_N)$;\\
		\hspace{-2mm}$\mathrm{null}( \cP - \cA \cP )=\mathrm{span}\{ \mathds{1}_N \otimes I_M \}$.\\
		\hspace{-2mm}\HRule\\
	\end{tabular}
\end{table}

%
%


%

\section{Penalized Formulation of Diffusion}\label{sec-diffusion-penalty}
In this section, we employ the properties derived in the previous section to reformulate the unconstrained optimization problem \eqref{prob-dist} into the equivalent constrained problem \eqref{prob-compact-2}, which will be solved using a penalized formulation.  This derivation will help clarify the origin of the $O(\mu^2_{\max})$ bias from \eqref{diffusion-limsup} in the standard diffusion implementation. 
%

\subsection{Constrained Problem Formulation}
\label{sbsec-diffusion-penalty}
To begin with, note that the unconstrained problem \eqref{prob-dist} is equivalent to the following constrained problem:
\eq{\label{prob-dist-2}
\min_{\{w_k\}}\quad& \sum_{k=1}^{N} q_k J_k(w_k), \nnb
\mathrm{s.t.}\quad& w_1=w_2=\cdots=w_N.
}
Now we introduce the block vector $\sw \define \mathrm{col}\{w_1,\cdots, w_N\}\in \RR^{NM}$ and 
\eq{\label{cJ-defi}
	\cJ^\star(\sw) \define \sum_{k=1}^{N}q_k J_k(w_k),
}
With \eqref{xcn86778} and \eqref{cJ-defi}, problem \eqref{prob-dist-2} is equivalent to
\eq{\label{prob-compact-1}
	\min_{\sw\in \RR^{NM}}\quad \cJ^\star(\sw), \quad
	\mathrm{s.t.}\quad \frac{1}{2}\left({\cP - \cA \cP}\right) \sw = 0.
}
%
%
%
%
From Lemma  \ref{coro-psd}, we know that $P - A P$ is symmetric and positive semi-definite. Therefore, we can decompose 
\eq{\label{syud}
\frac{P - AP}{2} = U \Sigma U\tran,
}
where $\Sigma \in \RR^{N\times N}$ is a non-negative diagonal matrix and $U\in \RR^{N\times N}$ is an orthogonal matrix. If we introduce the symmetric square-root matrix
\eq{\label{V-defi}
V \define U \Sigma^{1/2} U\tran \in \RR^{N\times N},
}
then it holds that
\eq{\label{nweh}
\frac{P - AP}{2} = V^2.
}
Let $\cV \define V\otimes I_M$ so that
\eq{\label{P-AP=VTV}
	\frac{\cP - \cA \cP}{2} = \cV^2.
}
{
\begin{lemma}[\sc Nullspace of $V$]\label{lm:null-V}
	With $V$ defined as in \eqref{V-defi}, it holds that
	\eq{\label{null-V}
	\mathrm{null}(V) = \mathrm{null}(P-AP) = \mathrm{span}\{\mathds{1}_N\}.
	}
\end{lemma}
\begin{proof}To prove $\mathrm{null}(V) = \mathrm{null}(P-AP)$, it is enough to prove 
	\eq{\label{xcn388}
	(P - AP) x = 0 \Longleftrightarrow Vx = 0.
	}
	Indeed, notice that
	\eq{\label{zxc0-1}
		(P - A P) x = 0
		\Rightarrow&\ V^2 x =0 \Rightarrow x\tran V\tran V x =0 \nnb
		\Rightarrow&\ \|V x\|^2 = 0 \Rightarrow  V x =0.
	}
	The reverse direction ``$\Leftarrow$" in \eqref{xcn388} is obvious.
\end{proof}
\vspace{1mm}
\noindent {\bf Remark 4. (Nullspace of $\cV$)}  
 Similar to the arguments in \eqref{xcn987} and \eqref{xcn86778}, we have
	\eq{\label{xcn987-2}
		\mathrm{null}( \cV ) =  \mathrm{null}( \cP - \cA \cP ) = \mathrm{span}\{ \mathds{1}_N \otimes I_M \},
	}
	and, hence,
	\eq{\label{zn}
		\cV \sx = 0 \Longleftrightarrow ( \cP - \cA \cP )\sx=0 \Longleftrightarrow x_1 =\cdots = x_N.
	}
	\rightline \qed
}
%
%
%
%
With \eqref{zn}, problem \eqref{prob-compact-1} is equivalent to
\eq{\label{prob-compact-2}
	\min_{\sw\in \RR^{NM}}\quad \cJ^\star(\sw), \quad 
	\mathrm{s.t.}\quad \cV \sw = 0.
}
In this way, we have transformed the original problem \eqref{prob-dist} to the equivalent constrained problem \eqref{prob-compact-2}.

\subsection{Penalized Formulation}

There are many techniques to solve constrained problems of the form \eqref{prob-compact-2}. One useful and popular technique is to add a penalty term to the cost function and to consider instead a penalized problem of the form:
%
\eq{\label{prob-penalty-1}
\min_{\sw\in \RR^{NM}}\quad \cJ^\star(\sw) + \frac{1}{\alpha} \left\|\cV \sw \right\|^2,
}
where $\alpha > 0$ is a penalty parameter. Problem \eqref{prob-penalty-1} is not equivalent to \eqref{prob-compact-2} but is a useful approximation. The smaller the value of $\alpha$ is, the closer the solutions of problems \eqref{prob-compact-2} and \eqref{prob-penalty-1} become to each other \cite{boyd2004convex,fletcher2013practical,towfic2014adaptive}. 
We now verify that the diffusion strategy \eqref{d-1}--\eqref{d-2} follows from applying an incremental technique to solving the approximate penalized problem \eqref{prob-penalty-1}, not the real problem \eqref{prob-compact-2}. It will then become clear that the diffusion estimate $w_{k,i}$ cannot converge to the exact solution $w^\star$ of problem \eqref{prob-dist} (or \eqref{prob-compact-2}).

 Since \eqref{P-AP=VTV} holds, problem \eqref{prob-penalty-1} is equivalent to
\eq{\label{prob-penalty-2}
\min_{w\in \RR^{NM}}\quad \cJ^\star(\sw) + \frac{1}{2\alpha}\sw\tran(\cP-\cA \cP)\sw.
}
This is an unconstrained problem, which we can solve using, for example, a  {diagonally-weighted} incremental algorithm, namely,
\begin{equation}
\left\{
\begin{aligned}
	\psi_i &= \sw_{i-1} - \alpha \cP^{-1} \grad \cJ^\star(\sw_{i-1}), \label{incremental-1} \\
	\sw_i &= \psi_i - \alpha \cP^{-1} \Big( \frac{1}{\alpha}(\cP-\cA \cP)\psi_i \Big), 
\end{aligned}
\right.
\end{equation}
The above recursion can be simplified as follows. Assume we select 
\eq{\label{alpha}
	\alpha\define \beta^{-1},
}
where $\beta$ is the same constant used in relation \eqref{q-A}. Recall from \eqref{beta-hastings}, \eqref{beta-averaging}, \eqref{beta-averaging-4} and \eqref{xchsss} that $\beta= O(1/\mu_{\max})$ and hence $\alpha = O(\mu_{\max})$. Moreover, from the definition of $\cJ^\star(\sw)$ in \eqref{cJ-defi}, we have
\eq{\label{grad-J-star}
	\grad \cJ^\star (\sw) = 
	\ba{c}
	q_1 \grad J_1(w_1)\\
	\vdots\\
	q_N \grad J_N(w_N)
	\ea	
}
Using \eqref{q-A}, namely,
\eq{\label{q_k-C-mu-p}
q_k = \beta \mu_k p_k,
}
we find that
%
%
\eq{\label{zlckj}
\alpha \cP^{-1} \grad \cJ^\star(\sw_{i-1}) = 
\ba{c}
\mu_1 \grad J_1(w_{1,i-1})\\
\vdots\\
\mu_N \grad J_N(w_{K,i-1})
\ea.
}
We further introduce the aggregate cost (which is similar to \eqref{cJ-defi} but without the weighting coefficients):
\eq{
\cJ^o(\sw) \define \sum_{k=1}^{N}J_k(w_k),
}
and note that
\eq{\label{grad-J^o}
\grad \cJ^o(\sw) =
\ba{c}
\grad J_1(w_1)\\
\vdots\\
\grad J_N(w_N)
\ea.	
}
Let $\cM\define \mbox{diag}\{\mu_1,\mu_2,\cdots,\mu_N\}\otimes I_M$. Using \eqref{zlckj} and \eqref{grad-J^o}, the first recursion in \eqref{incremental-1} can be rewritten as
\eq{\label{zcb-2}
\psi_i = \sw_{i-1} - \cM \grad \cJ^o(\sw_{i-1}).
}
For the second recursion of \eqref{incremental-1}, it can be rewritten as
\eq{\label{zcb-1}
	\sw_i = \cA\tran \psi_i
}
because $\cA \cP=\cP\cA\tran$. Relations \eqref{zcb-2}--\eqref{zcb-1} are equivalent to \eqref{d-1}--\eqref{d-2}. Specifically, if we collect all iterates from across all agents into block vectors $\{\sw_i, \psi_i\}$, then \eqref{d-1}--\eqref{d-2} would lead to \eqref{zcb-2}--\eqref{zcb-1}.  From this derivation, we conclude that the diffusion algorithm \eqref{d-1}--\eqref{d-2} can be interpreted as performing the diagonally-weighted incremental construction \eqref{incremental-1} to solve the approximate penalized problem \eqref{prob-penalty-2}. Since this construction is not solving the real problem \eqref{prob-dist}, there exists a bias between its fixed point and the real solution $w^\star$. As shown in \eqref{diffusion-limsup}, the size of this bias is related to $\mu_{\max}$. When $\mu_{\max}$ is small, the bias is also small. This same conclusion can be seen by noting that a small $\mu_{\max}$ corresponds to a large penalty factor $1/\alpha$ under which the solutions to problems \eqref{prob-dist} and \eqref{prob-compact-2} approach each other. 

\section{Development of Exact Diffusion}\label{sec-exact-diffusion}

{We now explain how to adjust the diffusion strategy \eqref{d-1}--\eqref{d-2} to ensure exact  convergence to $w^{\star}$. Instead of solving the approximate penalized problem \eqref{prob-penalty-2}, we apply the primal-dual saddle point method to solve the original problem \eqref{prob-compact-2} directly. 
We continue to assume that the combination policy $A$ is primitive and satisfies the local balancing condition \eqref{local-balance}. }

%


 To solve \eqref{prob-compact-2} with saddle point algorithm, we first introduce the augmented Lagrangian function:
\eq{
	\cL_a(\sw,\sy)
	=&\ \cJ^\star(\sw)+\frac{1}{\alpha}\sy\tran \cV \sw + \frac{1}{2\alpha}\left\|\cV\sw\right\|^2 \nnb
	\overset{\eqref{P-AP=VTV}}{=}&\ \cJ^\star(\sw) \hspace{-0.8mm}+\hspace{-0.8mm}\frac{1}{\alpha}\sy\tran \cV \sw \hspace{-0.8mm}+\hspace{-0.8mm} \frac{1}{4\alpha}\sw\tran (\cP\hspace{-0.8mm}-\hspace{-0.8mm}\cP\cA\tran) \sw,
}
where $\sy = \col\{y_1, \cdots, y_N\} \in \RR^{NM}$ is the dual variable. The standard primal-dual saddle point algorithm has recursions
\begin{equation}
	\left\{
	\begin{aligned}
	\sw_i &= \sw_{i-1} - \alpha \grad_{\sw} \cL_a(\sw_{i-1},\sy_{i-1}), \label{AL-1}\\
	\sy_i &= \sy_{i-1} + \alpha \left(\frac{1}{\alpha} \cV \sw_i\right) = \sy_{i-1} + \cV \sw_i.
	\end{aligned}
	\right.
\end{equation}
 The first recursion in \eqref{AL-1} is the primal descent while the second recursion is the dual ascent. Now, instead of performing the descent step directly as shown in the first recursion in \eqref{AL-1}, we perform it in an incremental manner. Thus, let 
\eq{
	\cD(\sw) \define \frac{1}{4\alpha}\sw\tran (\cP \hspace{-0.8mm}-\hspace{-0.8mm}\cP \cA\tran) \sw, \ \cC(\sw,\sy) \define \frac{1}{\alpha}\sy\tran \cV \sw,
}
so that
\eq{
	\cL_a(\sw,\sy_{i-1})=\cJ^\star(\sw) + \cD(\sw) + \cC(\sw,\sy_{i-1}).
}
The diagonally incremental recursion that corresponds to the first step in \eqref{AL-1} is then:
\begin{equation}
	\left\{
	\begin{aligned}
	\theta_i &= \sw_{i-1}-\alpha \cP^{-1} \grad \cJ^\star(\sw_{i-1}), \\
	\phi_i &= \theta_i - \alpha \cP^{-1} \grad \cD(\theta_i) = \frac{I_{MN} + \cA\tran}{2}\,\theta_i = \tcA \tran \theta_i, \label{i-2}\\
	\sw_i &= \phi_i \hspace{-0.8mm}-\hspace{-0.8mm} \alpha \cP^{-1} \grad_{\sw} \cC(\phi_i,\sy_{i-1})=\phi_i \hspace{-0.8mm}-\hspace{-0.8mm} \cP^{-1}\cV \sy_{i-1},
	\end{aligned}
	\right.
\end{equation}
%
%
where in the second recursion of \eqref{i-2} we introduced 
\eq{\label{tA}
\tcA \define (I_{MN}+\cA)/2.
}
We know from \eqref{A-eig-val} that the eigenvalues of $\overline{A}$ are positive and lie within the interval $(0,1]$. In \eqref{i-2}, if we substitute the first and second recursions into the third one, and also recall \eqref{zlckj} that $\alpha \cP^{-1}\grad \cJ^\star(\sw_{i-1}) = \cM \grad \cJ^o(\sw_{i-1})$, then we get
\eq{\label{zh239}
	\sw_i = \tcA\tran \Big(\sw_{i-1} \hspace{-0.8mm}-\hspace{-0.8mm} \cM\grad \cJ^o(\sw_{i-1})\Big)\hspace{-0.8mm}-\hspace{-0.8mm}\cP^{-1}\cV \sy_{i-1}.
}
Replacing {\color{black}the first recursion} in \eqref{AL-1} with \eqref{zh239}, the previous primal-dual saddle point recursion \eqref{AL-1} becomes
	\begin{equation}
	\boxed{
	\left\{
	\begin{aligned}
	\sw_i &= \tcA\tran \Big(\sw_{i\hspace{-0.3mm}-\hspace{-0.3mm}1}\hspace{-0.8mm}-\hspace{-0.8mm}\cM \grad \cJ^o(\sw_{i\hspace{-0.3mm}-\hspace{-0.3mm}1})\Big)\hspace{-0.8mm}-\hspace{-0.8mm}\cP^{-1}\cV \sy_{i-1} \label{zn-1}\\
	\sy_i &= \sy_{i-1} + \cV \sw_i
	\end{aligned}
	\right.}
	\end{equation}
Recursion \eqref{zn-1} is the primal-dual form of the exact diffusion recursion we are seeking. For the initialization, we set $y_{-1}=0$ and $\sw_{-1}$ to be any value, and hence for $i=0$ we have
\begin{equation}
\left\{
\begin{aligned}
\sw_0 &= \tcA\tran \Big(\sw_{\hspace{-0.3mm}-\hspace{-0.3mm}1}\hspace{-0.8mm}-\hspace{-0.8mm}\cM \grad \cJ^o(\sw_{\hspace{-0.3mm}-\hspace{-0.3mm}1})\Big), \label{zn-0}\\
\sy_0 &= \cV \sw_0. 
\end{aligned}
\right.
\end{equation}
 We can rewrite \eqref{zn-1} in a simpler form by eliminating the dual variable ${\sy}$ from the first recursion. For $i=1,2,\cdots,$ from \eqref{zn-1} we have
\eq{
	&\ \sw_i - \sw_{i-1} \nnb
	=&\ \tcA\tran\Big(\sw_{i-1}\hspace{-0.8mm}-\hspace{-0.8mm}\sw_{i-2} \hspace{-0.8mm}-\hspace{-0.8mm} \cM \big(\grad \cJ^o(\sw_{i-1}) \hspace{-0.8mm}-\hspace{-0.8mm} \grad \cJ^o(\sw_{i-2})\big) \Big) \nnb
	&\ -\cP^{-1}\cV (\sy_{i-1}-\sy_{i-2}). \label{znn9}
}
From the second step in \eqref{zn-1} we have 
\eq{\label{xn678}
&\ \cP^{-1}\cV(\sy_{i-1}-\sy_{i-2}) = \cP^{-1}\cV^2 \sw_{i-1} \nnb
\overset{\eqref{P-AP=VTV}}{=}&\ \cP^{-1}\left( \frac{\cP-\cP\cA\tran}{2} \right)\sw_{i-1} = \left(\frac{I_{MN} - \cA\tran}{2}\right)\sw_{i-1}.
}
Substituting \eqref{xn678} into \eqref{znn9}, we arrive at
\eq{\label{e-d}
	\boxed{
	\sw_i \hspace{-1mm}=\hspace{-1mm} \tcA\tran \Big(\hspace{-0.8mm} 2 \sw_{i\hspace{-0.3mm}-\hspace{-0.3mm}1} \hspace{-0.8mm}-\hspace{-0.8mm} \sw_{i\hspace{-0.3mm}-\hspace{-0.3mm}2} \hspace{-1mm}-\hspace{-1mm} \cM \big(\grad \cJ^o(\hspace{-0.3mm}\sw_{i\hspace{-0.3mm}-\hspace{-0.3mm}1}\hspace{-0.5mm}) \hspace{-0.8mm}-\hspace{-0.8mm} \grad \cJ^o(\sw_{i\hspace{-0.3mm}-\hspace{-0.3mm}2}) \big) \hspace{-0.8mm}\Big)}
}
Recursion \eqref{e-d} is the primal version of the exact diffusion.

We can rewrite \eqref{e-d} in a distributed form that resembles \eqref{d-1}--\eqref{d-2} more closely, as listed below in Algorithm 1, where we denote the entries of $\overline{A}$ by $\overline{a}_{\ell k}$. It is observed in Algorithm 1 that the exact diffusion strategy resembles \eqref{d-1}--\eqref{d-2} to great extent, with the addition of a ``correction'' step between the adaptation and combination step. In the correction step, the intermediate estimate $\psi_{k,i}$ is ``corrected" by removing from it the difference between $w_{k,i-1}$ and $\psi_{k,i-1}$ from the previous iteration. Moreover, it is also observed that the exact and standard diffusion strategies have essentially the same computational complexity, apart from $2M$ ($M$ is the dimension of $w_{k,i}$) additional additions per agent in the correction step of the exact implementation. {\color{black}Also, there is one combination step in each iteration, which reduces the communication cost by about one half in comparison to recent DIGing-based works \cite{lorenzo2016next, nedich2016achieving, qu2017harnessing, xu2015augmented, nedic2016geometrically}.

\begin{table}
	\noindent \HRule\\
	\noindent \textbf{\footnotesize Algorithm 1} {\footnotesize (Exact diffusion strategy for agent $k$)} \vspace{-2mm}\\ 
	\HRule\\
	\noindent \textbf{\vspace{0mm}{\scriptsize Setting:} 
	}
	{\scriptsize Let $\tA=(I_{N}+A)/2$, and $w_{k,\hspace{-0.2mm}-\hspace{-0.3mm}1}$ arbitrary. {\color{black}Set $\psi_{k,-1} = w_{k,-1}$.}}
	\noindent \textbf{\vspace{0mm}{\scriptsize {\color{white}Setting:}} 
	}
	{\scriptsize {\color{black}Let $\mu_k = q_k \mu_o/p_k$}. }

	\noindent 
		\textbf{\vspace{-4mm} \hspace{-1.3mm}{\scriptsize Repeat for $i=0,1,2,\cdots$}}\\
		{\footnotesize
			\eq{
				\psi_{k,i} &= w_{k,i-1} - \mu_k \grad J_k(w_{k,i-1}), \hspace{5mm} \mbox{\footnotesize (adaptation)} \label{adapt}\\
				\phi_{k,i} &= \psi_{k,i} + w_{k,i-1}  - \psi_{k,i-1}, \hspace{8mm} \mbox{\footnotesize (correction)} \label{correct}\\
				w_{k,i} &= \sum_{\ell\in \cN_k} \overline{a}_{\ell k} \phi_{\ell,i}. \hspace{2.32cm} \mbox{\footnotesize (combination)}  \label{combine}
			}}
			\HRule
			\vspace{-8mm}
		\end{table}

One can directly run Algorithm 1 when the Perron entries $\{p_k\}$ are known beforehand, as explained in Section II-B. When this is not the case, we can blend iteration \eqref{power_iteartion} into the algorithm and modify it as follows.
}

{\color{black}
\begin{table}[h]
	\noindent \HRule\\
	\noindent \textbf{\footnotesize Algorithm 1'} {\footnotesize (Exact diffusion strategy when $p$ is unknown)} \vspace{-2mm}\\ 
	\HRule\\
	\noindent \textbf{\vspace{0mm}{\scriptsize Setting:} 
	}
	{\scriptsize Let $\tA=(I_{N}+A)/2$, and $w_{k,\hspace{-0.2mm}-\hspace{-0.3mm}1}$ arbitrary. {\color{black}Set $\psi_{k,-1} = w_{k,-1}$,}}
	
	{\scriptsize \hspace{1.06cm} and $z_{k,-1} = e_k$.}
	\vspace{1mm}
	
	\noindent 
	\textbf{\vspace{-4mm} \hspace{-1.3mm}{\scriptsize Repeat for $i=0,1,2,\cdots$}}\\
	{\footnotesize
		\eq{
			z_{k,i} &= \sum_{\ell\in \cN_k} \bar{a}_{\ell k} z_{\ell,i-1}, \hspace{2.2cm} \mbox{\footnotesize (power iteration)} \label{power-iteration} \\
			\psi_{k,i} &= w_{k,i-1} - \frac{q_k \mu_o}{z_{k,i}(k)} \grad J_k(w_{k,i-1}), \hspace{1mm} \mbox{\footnotesize (adaptation)} \label{adapt-d}\\
			\phi_{k,i} &= \psi_{k,i} + w_{k,i-1}  - \psi_{k,i-1}, \hspace{10.5mm} \mbox{\footnotesize (correction)} \label{correct-d}\\
			w_{k,i} &= \sum_{\ell\in \cN_k} \overline{a}_{\ell k} \phi_{\ell,i}. \hspace{2.57cm} \mbox{\footnotesize (combination)}  \label{combine-d}
		}}
		\HRule
	\end{table}
}

\section{Significance of Balanced Policies}
\label{sec-general-A}
The stability and convergence properties of the exact diffusion strategy \eqref{adapt}--\eqref{combine}  will be examined in detail in Part II \cite{yuan2017exact2}. There we will show that exact diffusion is guaranteed to converge for all balanced left-stochastic matrices for sufficiently small step-sizes. The local balancing property turns out to be critical in the sense that convergence may or may not occur if we move beyond the set of balanced policies. We can illustrate these possibilities here by means of examples. The two examples discussed in the sequel highlight the importance of having balanced combination policies for exact convergence.

Thus, consider the primal recursion of the exact diffusion algorithm \eqref{e-d}, where $\tcA$ is a general left-stochastic matrix. We subtract $\sw^\star$ from both sides of \eqref{e-d}, to get the error recursion
\eq{\label{xn266}
\twd_i = \tcA\tran \left( 2\twd_{i-1} - \twd_{i-2} {+} \cM \big(\grad \cJ^o(\hspace{-0.3mm}\sw_{i\hspace{-0.3mm}-\hspace{-0.3mm}1}\hspace{-0.5mm}) \hspace{-0.8mm}-\hspace{-0.8mm} \grad \cJ^o(\sw_{i\hspace{-0.3mm}-\hspace{-0.3mm}2}) \right),
}
where $\twd_i = \sw^\star - \sw_i$. 
When $\grad J_k(w)$ is twice-differentiable, we can appeal to the mean-value theorem from Lemma D.1 in \cite{sayed2014adaptation}, which allows us to express each difference
\eq{
	&\hspace{-10mm} \grad J_k(w_{k,i-1}) - \grad J_k(w^\star)  \nnb
	=&\ -\left( \int_0^1 \grad^2 J_k\big(w^\star \hspace{-0.8mm}-\hspace{-0.8mm} r \widetilde{w}_{k,i-1}\big)dr\right)\widetilde{w}_{k,i-1}.
}
If we let 
\eq{\label{H_k_i-1}
	H_{k,i-1} \hspace{-1.5mm}\define \hspace{-1.5mm} \int_0^1 \grad^2 J_k\big(w^\star \hspace{-0.5mm}-\hspace{-0.5mm} r\widetilde{w}_{k,i-1}\big)dr \in \RR^{M\times M},
}
and introduce the block diagonal matrix:
\eq{\label{H_i-1}
	\cH_{i-1} \hspace{-1.5mm}\define \hspace{-1.5mm} \mathrm{diag}\{H_{1,i-1},H_{2,i-1},\cdots,H_{N,i-1}\},
}
then we can rewrite
\eq{
	\grad \cJ^o(\sw_{i-1}) - \grad \cJ^o(\sw^\star) = - \cH_{i-1} \twd_{i-1}.\label{xcnh}
}
Notice that
\eq{\label{cnwh99}
&\hspace{-1cm} \grad  \cJ^o(\hspace{-0.3mm}\sw_{i\hspace{-0.3mm}-\hspace{-0.3mm}1}\hspace{-0.5mm}) \hspace{-0.8mm}-\hspace{-0.8mm} \grad \cJ^o(\sw_{i\hspace{-0.3mm}-\hspace{-0.3mm}2}) \nnb
=&\ \grad \cJ^o(\hspace{-0.3mm}\sw_{i\hspace{-0.3mm}-\hspace{-0.3mm}1}\hspace{-0.5mm}) \hspace{-0.8mm}-\hspace{-0.8mm} \grad \cJ^o(\sw^\star) \hspace{-0.8mm}+\hspace{-0.8mm} \grad \cJ^o(\sw^\star)  \hspace{-0.8mm}-\hspace{-0.8mm}
 \grad \cJ^o(\sw_{i\hspace{-0.3mm}-\hspace{-0.3mm}2})\nnb
 \overset{\eqref{xcnh}}{=}&\ \cH_{i-2}\twd_{i-2} - \cH_{i-1}\twd_{i-1}.
}
Combining \eqref{xn266}, \eqref{cnwh99} and the fact $\twd_{i-1}=\twd_{i-1}$, we have 
\eq{\label{xcn288sss}
\ba{c}
\twd_i\\
\twd_{i-1}
\ea
=
(\cF {-} \cG_{i-1} ) 
\ba{c}
\twd_{i-1}\\
\twd_{i-2}
\ea,
}
where
\eq{\label{cF}
\cF& \hspace{-1mm}\define\hspace{-1mm}
\ba{cc}
2\tcA\tran & -\tcA\tran\\
\cI_{MN} & 0
\ea \in \RR^{2MN \times 2MN},\\
\cG_{i-1}&\hspace{-1mm}\define\hspace{-1mm}
\ba{cc}
\tcA\tran \cM \cH_{i-1} & -\tcA\tran \cM \cH_{i-2} \\
0 & 0
\ea \in \RR^{2MN \times 2MN}. \label{cG}
}

In the next two examples, we consider the simple case where the dimension $M=1$, $q_k=1$ for $k\in\{1, \cdots, N\}$,  and the step-size $\cM$ = $\mu P^{-1}$, where 
\eq{
P = \diag\{p_1, \cdots, p_N\}\in \RR^{N\times N}.
}
In this situation, the matrix $\cF - \cG_{i-1}$ reduces to
\eq{\label{n2g88}
\cF\hspace{-1mm}-\hspace{-1mm}\cG_{i-1} \hspace{-1mm}= \hspace{-1mm}
\ba{cc}
\hspace{-1mm}
\tA\tran (2 I_N \hspace{-0.8mm}-\hspace{-0.8mm} \mu P^{-1} H_{i-1} )\hspace{-0.5mm} & \hspace{-0.5mm}-\tA\tran (I_N \hspace{-0.8mm}-\hspace{-0.8mm} \mu P^{-1} H_{i-2})\hspace{-2mm}\\
\hspace{-1mm}I_N \hspace{-0.5mm}&\hspace{-0.5mm} 0\hspace{-1mm}
\hspace{-2mm}
\ea.
}
{\color{black}Moreover, we also assume $H_i$ is iteration independent, i.e., 
\eq{
H_i = H, \quad \forall\ i = 1, 2, \cdots 
}
This assumption holds for quadratic costs $J_k(w)$. Under the above conditions, we have	
}
\eq{\label{23bb9}
(\cF - \cG_{i-1}) 
\ba{c}
\mathds{1}_N\\
\mathds{1}_N
\ea
=
\ba{c}
\tA\tran \mathds{1}_N\\
\mathds{1}_N
\ea
=
\ba{c}
\mathds{1}_N\\
\mathds{1}_N
\ea,
}
which implies that $\lambda_1 =1$ is one eigenvalue of $\cF - \cG_{i-1}$ no matter what the step-size $\mu$ is. However, since $\sw_0$ is initialized as $\cV \sy_0$ and, hence, lies in $\mathrm{range}(\cV)$, the eigenvalue $\lambda_1=1$ will not influence the convergence of recursion \eqref{xcn288sss} (the detailed explanation is spelled out in Sections II and III of Part II\cite{yuan2017exact2}). Let $\{\lambda_k\}_{k=2}^{2N}$ denote the remaining eigenvalues of $\cF-\cG_{i-1}$, and introduce
\eq{
\rho(\cF - \cG_{i-1}) \define \max\{ |\lambda_2|,|\lambda_3|,\cdots,|\lambda_{2N}| \}.
} 
It is $\rho(\cF - \cG_{i-1})$ that determines the convergence of recursion \eqref{xcn288sss}: the exact diffusion recursion \eqref{xcn288sss} will diverge if $\rho(\cF - \cG_{i-1}) >1$, and will converge if $\rho(\cF - \cG_{i-1}) < 1$.

\vspace{1mm}
{\color{black}
\noindent \textbf{Example 1 (Diverging case).} Consider the following left-stochastic matrix $A$:
\eq{\label{A-diverging}
A =
\ba{cccc}
0 & 0 & 0 & 1 \\
0 & 0.5 & 0.5 & 0\\
1 & 0 & 0.5 & 0\\
0 & 0.5 & 0 & 0
\ea.
}
It can be verified that $A$ is primitive, left-stochastic but not balanced. For such $A$, its Perron eigenvector $p$ can be calculated in advance, and hence $P$ is also known. 
Also, $H_{i-1}$ is assumed to satisfy
\eq{\label{H-diverging}
P^{-1} H_{i-1} = \diag\{20, 1, 1, 1\} \in \RR^{4 \times 4}
}
{Substituting the above $A$ and $PH_{i-1}$ into $\cF-\cG_{i-1}$ shown in \eqref{n2g88}, it can be verified that
\eq{\label{xcn23h}
	\rho(\cF - \cG_{i-1}) > 1
}
for {\em any} step-size $\mu > 0$. The proof is given in Appendix \ref{app-unstable-A} by appealing to the Jury test for stability. In the top plot in Fig. \ref{fig:general_A}, we show the spectral radius $\rho(\cF-\cG_{i-1})$ for step-sizes $\mu\in[1e^{-6},3]$. It is observed that $\rho(\cF-\cG_{i-1})>1$.}

{\color{black}By following similar arguments, we can find a counter example such that EXTRA will also diverge for any step-size $\mu > 0$, even if we assume the Perron eigenvector $p$ is known in advance. For example, if
\eq{\label{A-diverging-extra}
	A =
	\ba{ccccc}
	0.36 & 0.99 & 0 & 0 & 0 \\
	0 & 0.01 & 0 & 0.6 & 0\\
	0 & 0 & 0.02 & 0 & 0.95\\
	0 & 0 & 0.98 & 0.4 & 0\\
	0.64 & 0 & 0 & 0 & 0.05
	\ea \in \RR^{5\times 5}
}
and 
\eq{\label{H-diverging-extra}
	P^{-1} H_{i-1} = \diag\{20, 1, 1, 1, 1\} \in \RR^{5\times 5},
}
one can verify that EXTRA will diverge for any $\mu>0$ by following the arguments in Appendix \ref{app-unstable-A}. As a result, the push-sum based algorithms \cite{zeng2015extrapush,xi2015linear} that extend EXTRA to non-symmetric networks cannot always converge. This example indicates that the stability range $(c_{\rm lower}, c_{\rm upper})$ provided in \cite{zeng2015extrapush,xi2015linear} may not always be feasible. \\
\rightline \qed
}


}
\vspace{1mm}
\noindent \textbf{Example 2 (Converging case).} Consider the following left-stochastic matrix $A$:
\eq{\label{A-converging}
	A =
	\ba{ccccc}
	0.3 & 0.6 & 0.2 & 0 & 0\\
	0.2 & 0.2 & 0 & 0.3 & 0\\
	0.1 & 0.1 & 0.5 & 0.3 & 0.2\\
	0 & 0.1 & 0.3 & 0.4 & 0.1\\
	0.4 & 0 & 0 & 0 & 0.7
	\ea.
}
It can be verified that $A$ is primitive and not balanced. Also, $H_{i-1}$ is assumed to satisfy
\eq{\label{H-converging}
	P^{-1} H_{i-1}=\mathrm{diag}\{10, 10, 10, 10, 10\} \in \RR^{5\times 5}.
}
Substituting the above $A$ and $P^{-1} H_{i-1}$ into \eqref{n2g88}, {\color{black}it can be verified that $\rho(\cF) = 0.9923$. Therefore, when $\mu$ is sufficiently small, $\cF$ will dominate in $\cF - \cG_{i-1}$ and $\rho(\cF - \cG_{i-1}) < 1$. The simulations in Fig. \ref{fig:general_A_converging} confirm this fact. In particular, it is observed that $\rho(\cF - \cG_{i-1}) < 1$ when $\mu < 0.2$. As a result, the exact diffusion will converge when $\mu<0.2$ under this setting. }
\newline
\rightline \qed

{
\section{Numerical Experiments}\label{sec-simulation}
In this section we illustrate the performance of the proposed exact diffusion algorithm. In all figures, the $y$-axis indicates the relative error, i.e.,  $\|\sw_i-\sw^o\|^2/\|\sw_0-\sw^o\|^2$, where $\sw_i=\col\{w_{1,i},\cdots,w_{N,i}\}\in \RR^{NM}$ and $\sw^o=\col\{w^o,\cdots,w^o\}\in \RR^{NM}$. 

\subsection{Distributed Least-squares}\label{subsec:expe-ls}
In this experiment, we focus on solving the least-squares problem over the network shown in \ref{fig:networktopology}:
\eq{\label{prob-ls}
w^o = \argmin_{w\in \RR^M}\quad \frac{1}{2}\sum_{k=1}^{N} \|U_k w - d_k\|^2.
}
where the network size $N=20$ and the dimension $M=30$. Each entry in both $U_k\in \RR^{50\times 30}$ and $d_k\in \RR^{50}$ is generated from the standard Gaussian distribution $\cN(0,1)$. 
\begin{figure}
\centering
\includegraphics[scale=0.4]{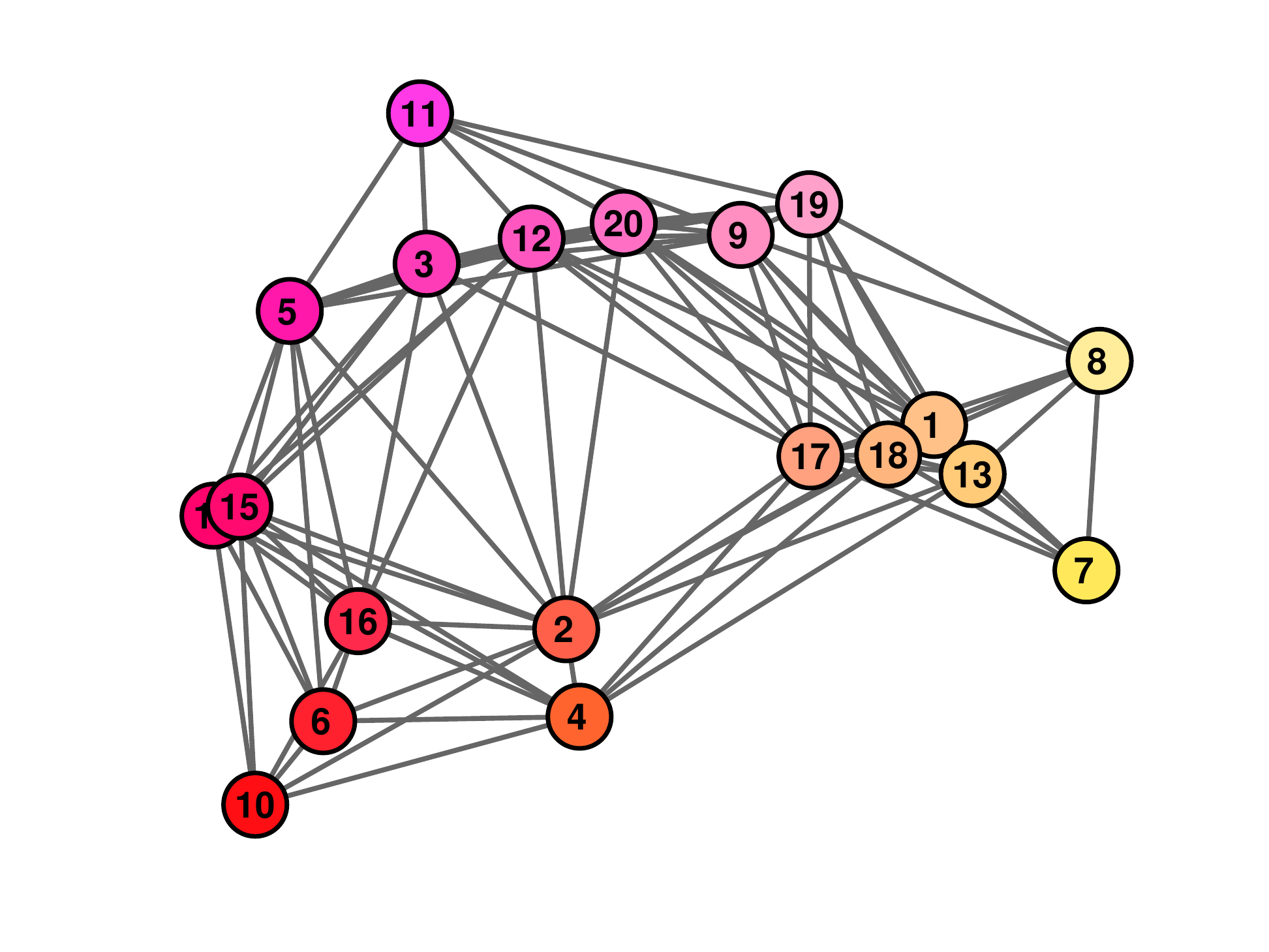}
\caption{\footnotesize Network topology used in the simulations.}
\label{fig:networktopology}
\end{figure}

We compare the convergence behavior of standard diffusion and the exact diffusion algorithm in the simulation. The left-stochastic matrix $A$ is generated through the averaging rule (see \eqref{a-lk-ave}), and each agent $k$ employs step-size $\mu_k=\mu_o/n_k$ (see \eqref{mu-ave}) where $\mu_o$ is a small constant step-size. The convergence of both algorithms is shown in Fig. \ref{fig:CS}, where we set $\mu_o=0.01$. It is observed that the standard diffusion algorithm converges to a neighborhood of $w^o$ on the order $O(\mu_o^2)$, while the exact diffusion converges exponentially fast to the exact solution $w^o$. This figure confirms that exact diffusion corrects the bias in standard diffusion.
\begin{figure}
	\centering
	\includegraphics[scale=0.4]{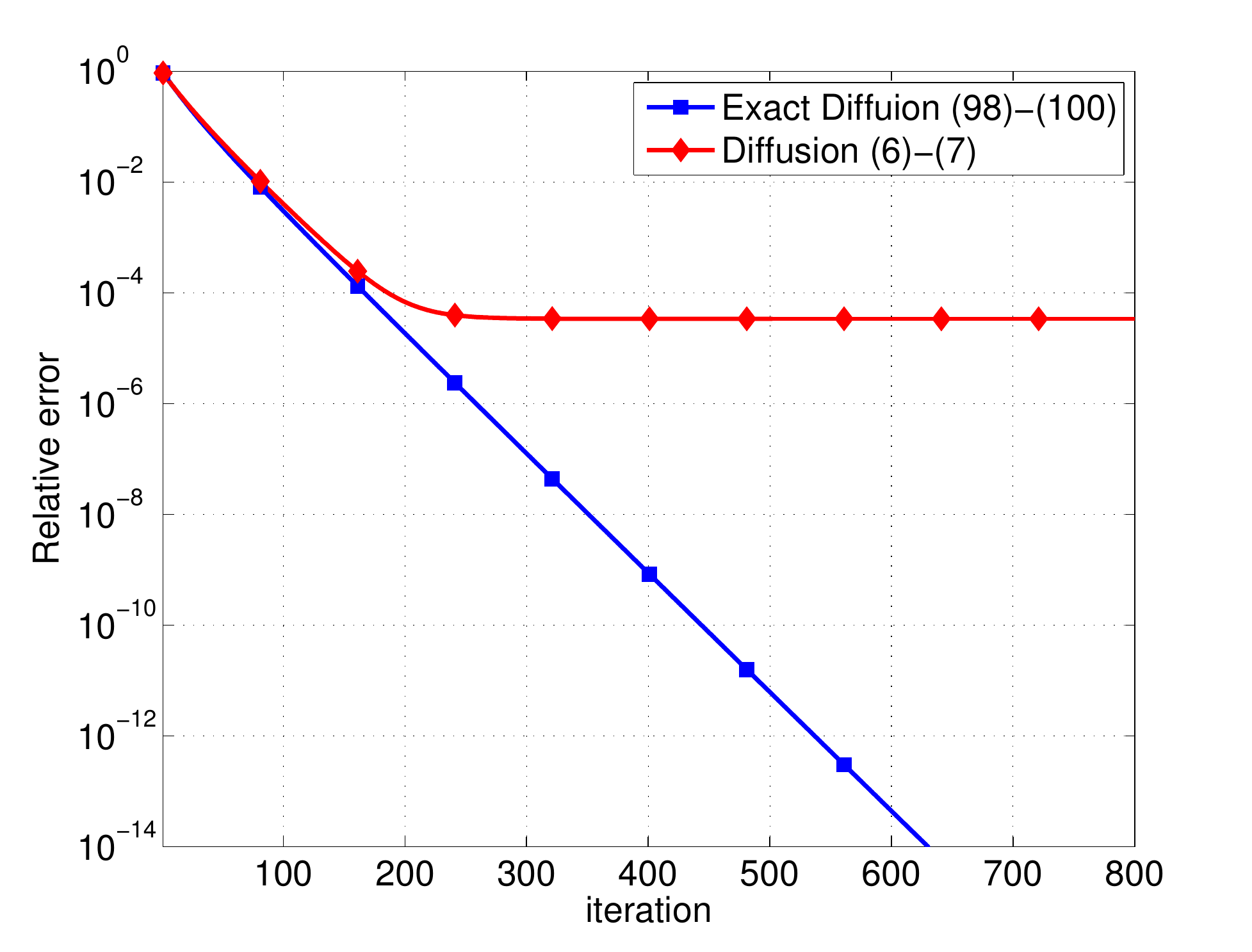}
	\caption{\footnotesize Convergence comparison between standard diffusion and exact diffusion for the distributed least-squares \eqref{prob-ls}.}
	\label{fig:CS}
\end{figure}

}

\subsection{Distributed Logistic Regression}\label{sec-dlr}
We next consider a pattern classification scenario. Each agent $k$ holds local data samples $\{h_{k,j}, \gamma_{k,j}\}_{j=1}^L$, where $h_{k,j}\in \RR^M$ is a feature vector and $\gamma_{k,j}\in \{-1,+1\}$ is the corresponding label. Moreover, the value $L$ is the number of local samples at each agent. All agents will cooperatively solve the regularized logistic regression problem over the network in Fig. \ref{fig:networktopology}:
\eq{\label{prob-lr}
w^o=\argmin_{w\in \RR^M} \sum_{k=1}^{N} \Big[\frac{1}{L}\sum_{\ell=1}^{L}\ln\big(1\hspace{-1mm}+\hspace{-1mm}\exp(-\gamma_{k,\ell} h_{k,\ell}\tran w)\big) \hspace{-1mm}+\hspace{-1mm} \frac{\rho}{2}\|w\|^2 \Big].
}
In the experiments, we set $N=20$, $M=30$, and $L=50$. For local data samples $\{h_{k,j}, \gamma_{k,j}\}_{j=1}^L$ at agent $k$, each
$h_{k,j}$ is generated from the standard normal distribution $\cN(0; 10I_M)$. To generate $\gamma_{k,j}$, we first generate an auxiliary random vector $w_0\in \RR^{M}$ with each entry following $\cN(0,1)$. Next, we generate $\gamma_{k,j}$ from a uniform distribution $\cU(0,1)$. If $\gamma_{k,j} \le 1/[1+\exp(-(h_{k,j})\tran w_0)]$ then $\gamma_{k,j}$ is set as $+1$; otherwise $\gamma_{k,j}$ is set as $-1$. We set $\rho=0.1$. 

We still compare the convergence behavior of the standard diffusion and exact diffusion. The left-stochastic matrix $A$ is generated through the averaging rule, and each agent $k$ employs step-size $\mu_k=\mu_o/n_k$. The convergence of both algorithms is shown in Fig. \ref{fig:lr-1}. The step-size $\mu_o=0.05$. It is also observed that the exact diffusion corrects the bias in standard diffusion.
\begin{figure}
	\centering
	\includegraphics[scale=0.4]{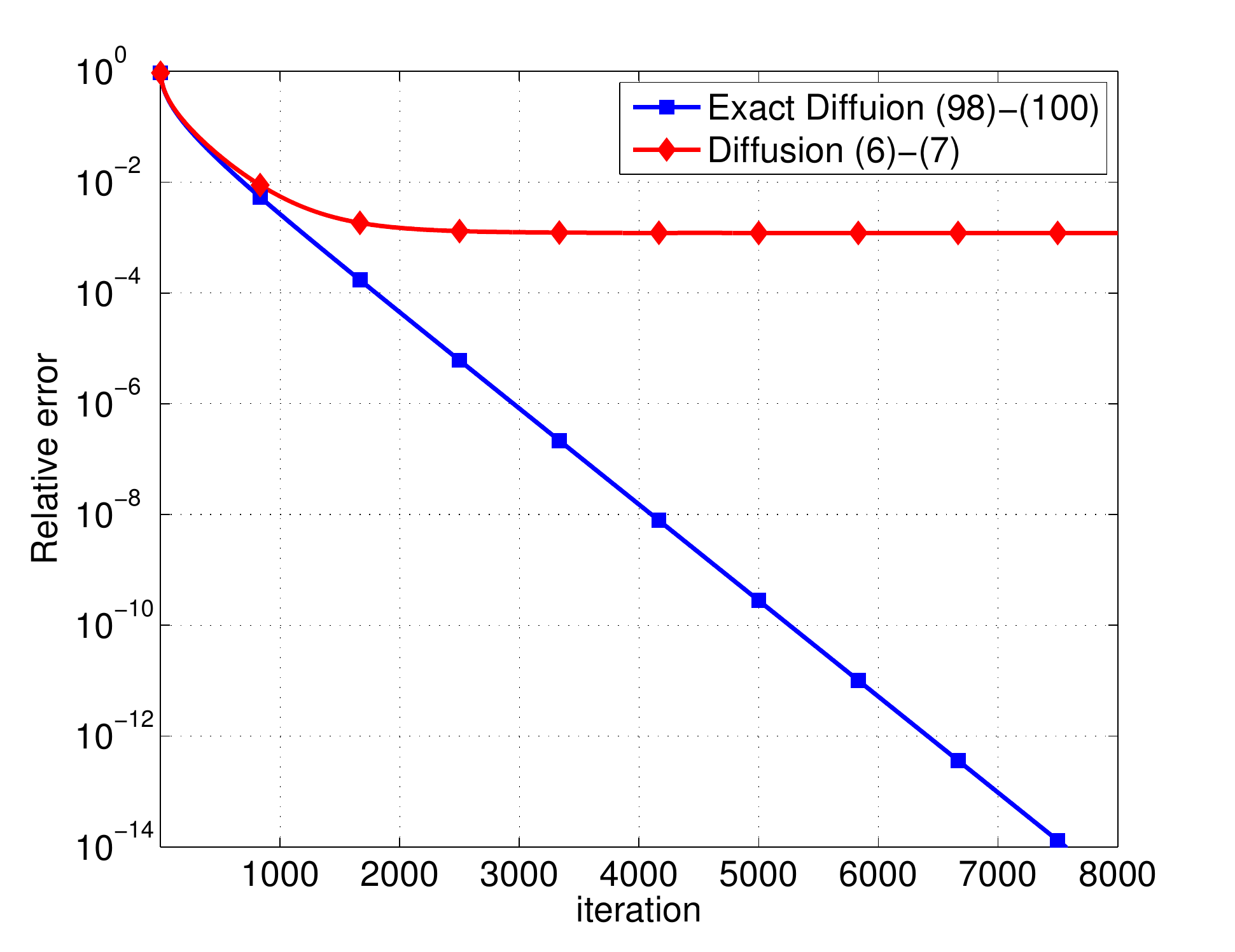}
	\caption{\footnotesize Convergence comparison between standard diffusion and exact diffusion for distributed logistic regression \eqref{prob-lr}.}
	\label{fig:lr-1}
\end{figure}


{\color{black}
\subsection{Averaging Rule v.s. Doubly Stochastic Rules}\label{sec-average-vs-ds}
In this subsection we test the convergence performance of exact diffusion under different combination matrices. Consider a network with a highly unbalanced topology as shown in Fig. \ref{fig:unbalanced}. Nodes $1$ and $2$ are ``celebrities'' with many neighbors, while the other $18$ nodes just have two neighbors each. Such a network topology is quite common over social networks. 

Interestingly, both the maximum degree rule and the Metropolis rule will generate the same doubly-stochastic combination matrix for this network. Let $L$ be the Laplacian matrix associated with that network, then the generated doubly-stochastic combination matrix is  
\eq{
A = I - L/19. 
}
This combination matrix $A$ merges information just slightly better than the identity matrix $I$ because the term $L/19$ is quite small, which is not efficient. In contrast, the normal agent $k$ (where $3\le k\le 20$) will assign $1/3$ to incoming information from agents $1$ and $2$ if the averaging rule is used, which combines information more efficiently and hence leads to faster convergence. In Fig. \ref{fig:highly_unbalanced}, we compare these two combination matrices over the distributed least-square problem \eqref{prob-ls}. The step-sizes are carefully chosen such that each combination matrix reaches its fastest convergence. As expected, it is observed that the averaging rule is almost three times faster than the doubly-stochastic rule.

\begin{figure}
	\centering
	\includegraphics[scale=0.55]{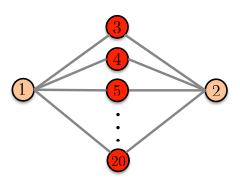}
	\caption{\footnotesize A highly unbalanced network topology.}
	\label{fig:unbalanced}
\end{figure}

\begin{figure}
	\centering
	\includegraphics[scale=0.4]{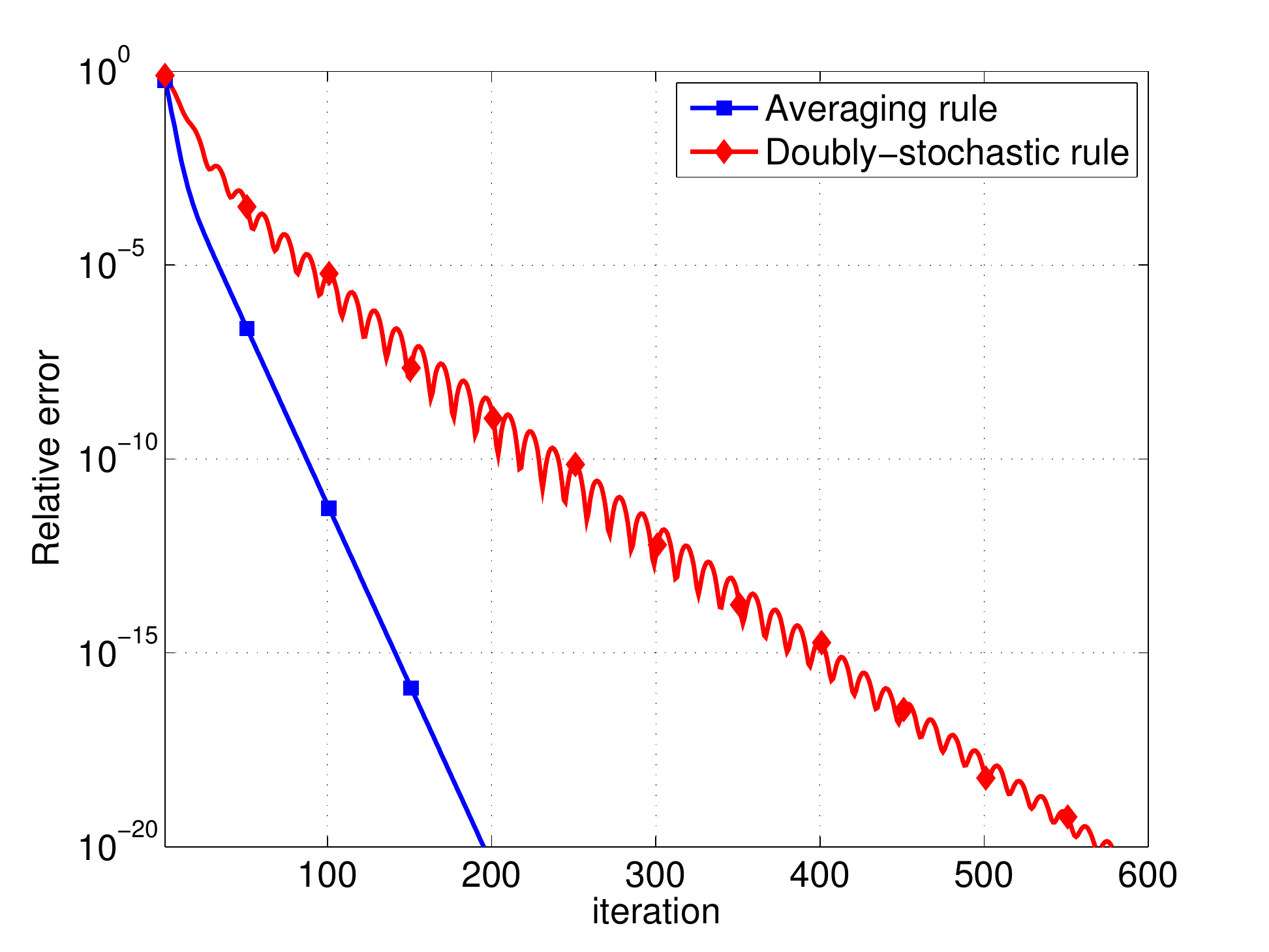}
	\caption{\footnotesize Convergence comparison between averaging rule and doubly-stochastic rule for distributed least-squares \eqref{prob-ls}.}
	\label{fig:highly_unbalanced}
\end{figure}

%
%
}

{
\subsection{Exact Diffusion for General Left-Stochastic $A$}\label{subsec-simulation-general-A}
In this subsection we test exact diffusion for the general left-stochastic $A$ shown in Section \ref{sec-general-A}. 
	In Fig. \ref{fig:general_A} we test the setting of Example $1$ in which $A$ is in the form of \eqref{A-diverging} and $H$ is \eqref{H-diverging}. We introduce $\rho= \rho(\cF-\cG_{i-1})$. In the top plot, we illustrate how $\rho$ varies with step-size $\mu$. {In this plot, the step-size varies over $[10^{-6}, 3]$, and the interval between two consecutive $\mu$ is $10^{-6}$. It is observed that $\rho>1$ for any $\mu\in [10^{-6}, 3]$, which confirms with our conclusion that exact diffusion will diverge for any step-size $\mu$ under the setting in Example $1$.} In the bottom plot of Fig. \ref{fig:general_A} we illustrate the standard diffusion converges to a neighborhood of $w^o$ on the order of $O(\mu^2)$ for $\mu=0.01$, while the exact diffusion diverges.
	
	 In Fig. \ref{fig:general_A_converging} we test the setting of Example $2$ in which $A$ is in the form of \eqref{A-converging} and $H$ is of \eqref{H-converging}. In the top plot, we illustrate how $\rho$ varies with $\mu$. It is observed that $\rho < 1$ when $\mu<0.2$, which implies that the exact diffusion recursion \eqref{xcn288sss} will converge when $\mu<0.2$. In the bottom figure, with $\mu=0.001$ it is observed that exact diffusion will converge exactly to $w^o$. Figures. \ref{fig:general_A} and \ref{fig:general_A_converging} confirm that general left-stochastic $A$ cannot always guarantee convergence to $w^o$.}	

\begin{figure}
	\centering
	\includegraphics[scale=0.38]{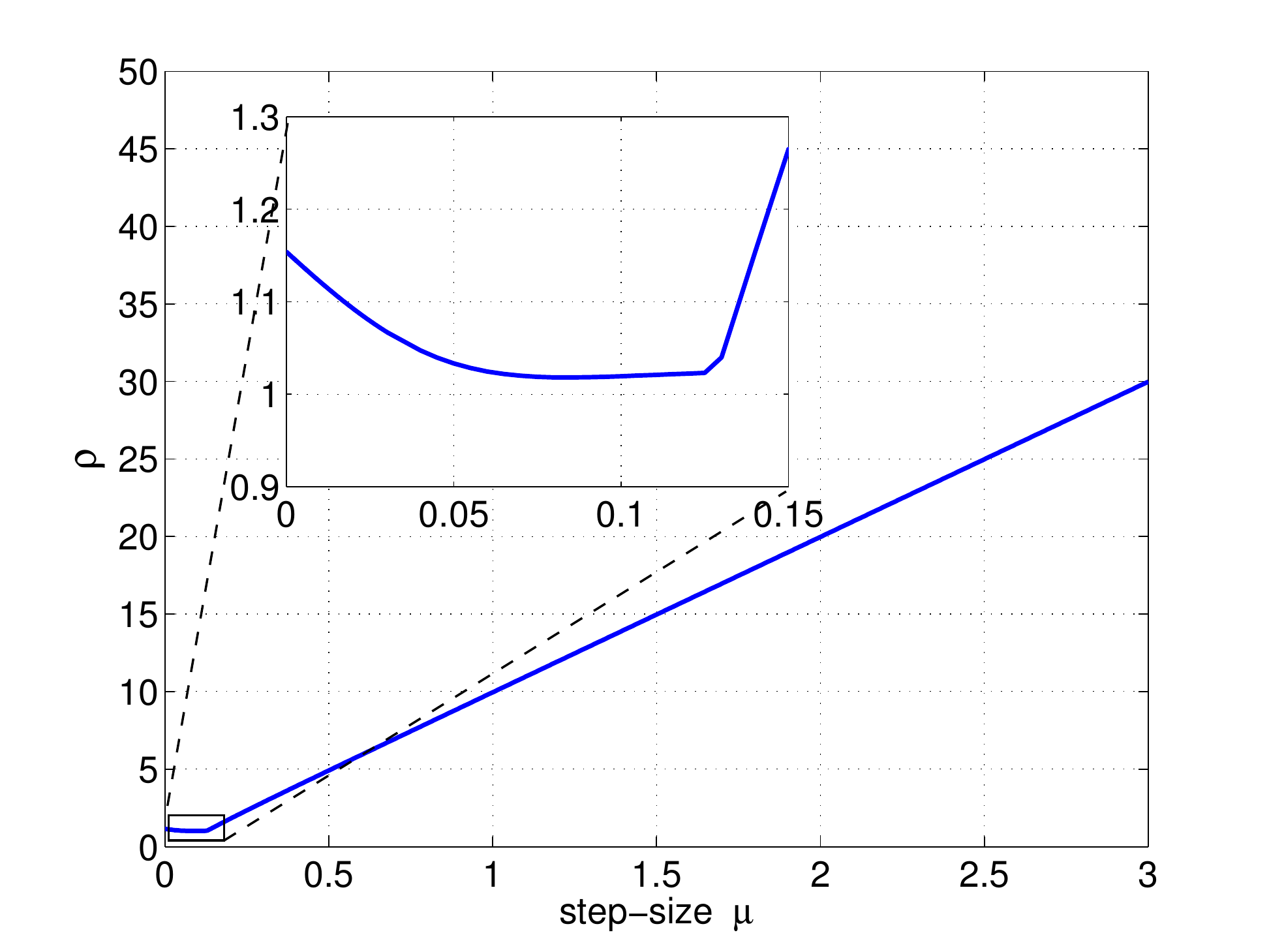}
	\includegraphics[scale=0.38]{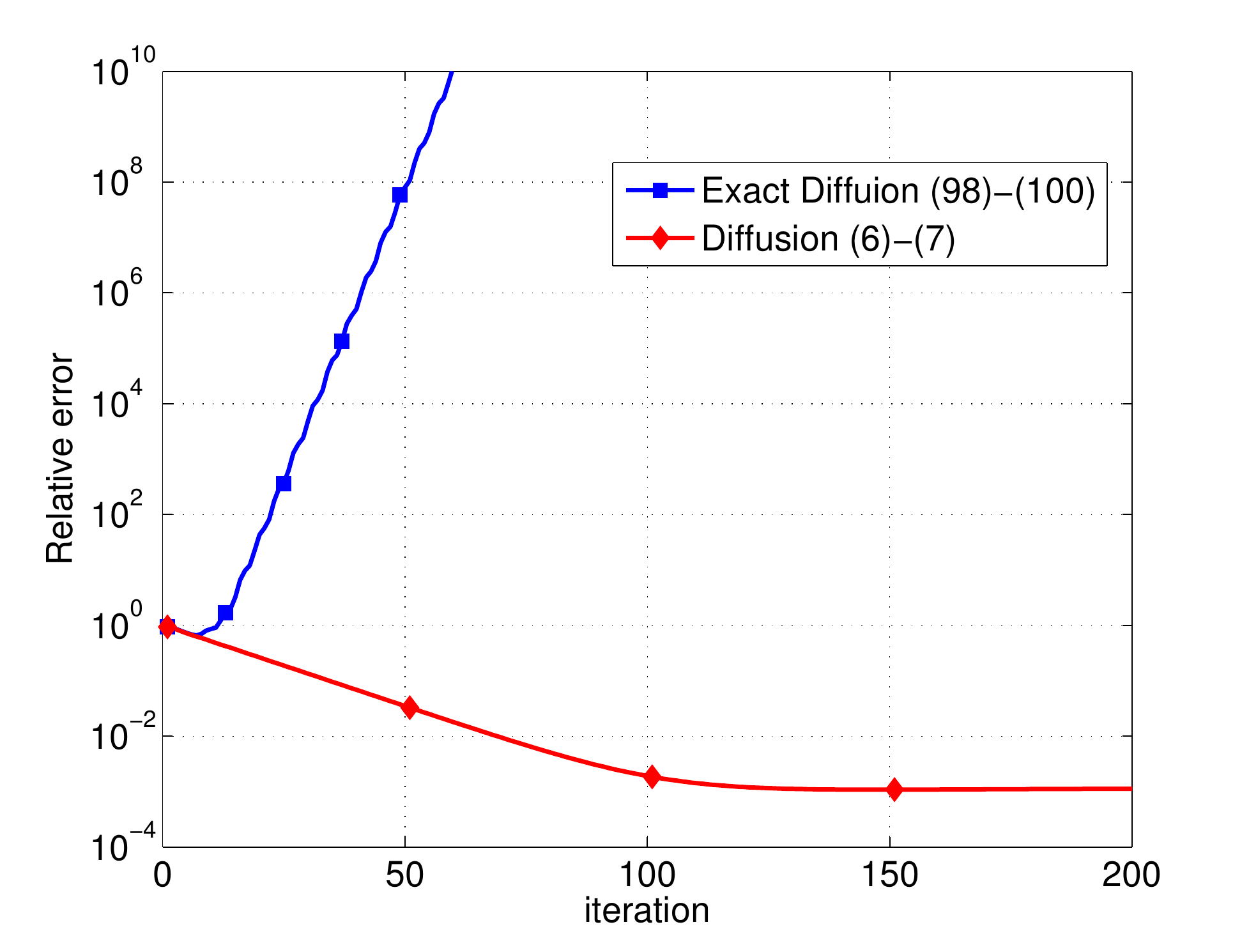}
	\vspace{-3mm}
	\caption{\footnotesize {Exact diffusion under the setting of Example 1 in Section \ref{sec-general-A}. Top: $\rho>1$ no matter what value $\mu$ is. Bottom: Convergence comparison between diffusion and exact diffusion when $\mu=0.01$.
		}}
		\label{fig:general_A}
	\end{figure}
	
	\begin{figure}
		\centering
		\includegraphics[scale=0.38]{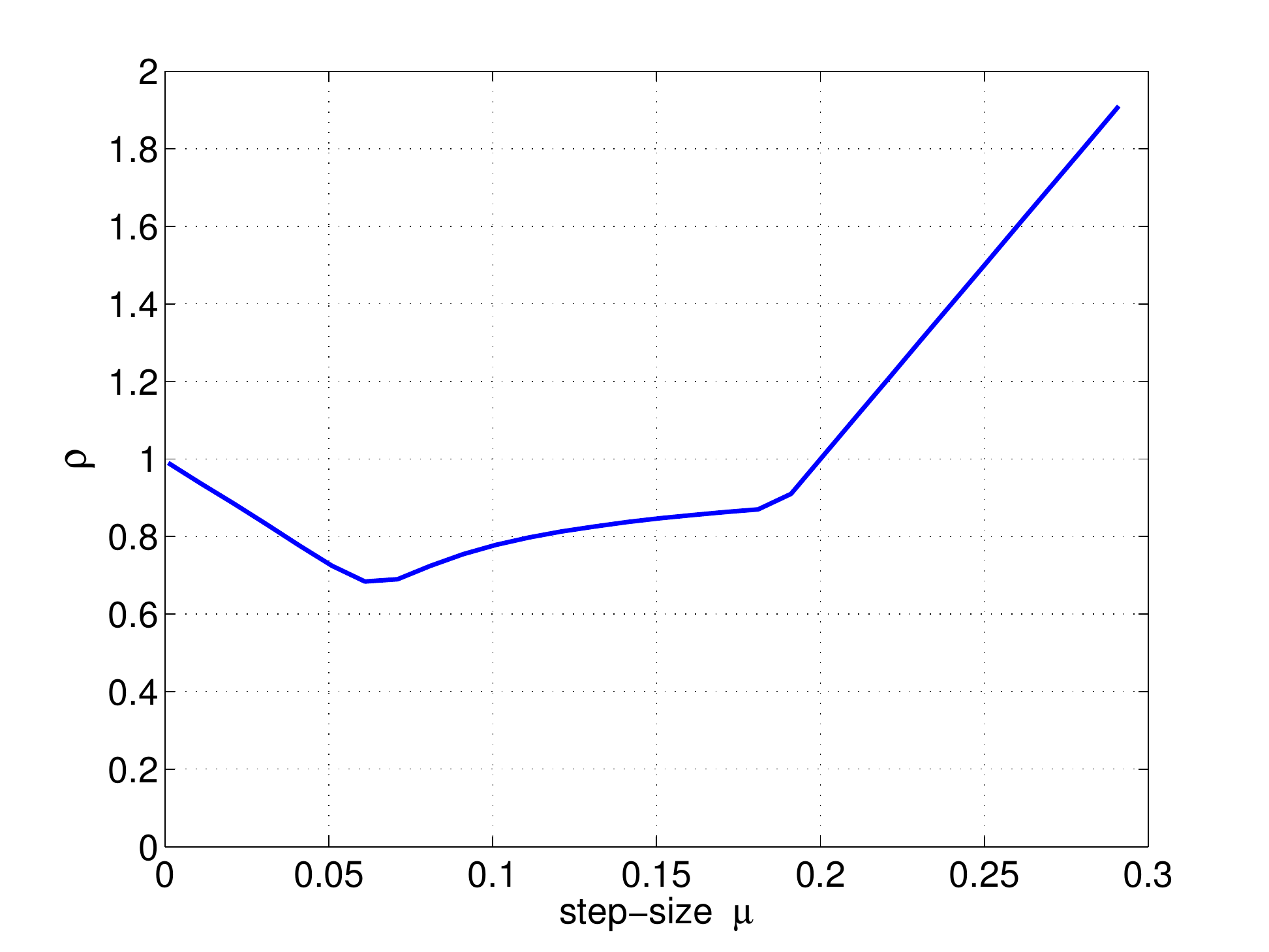}
		\includegraphics[scale=0.38]{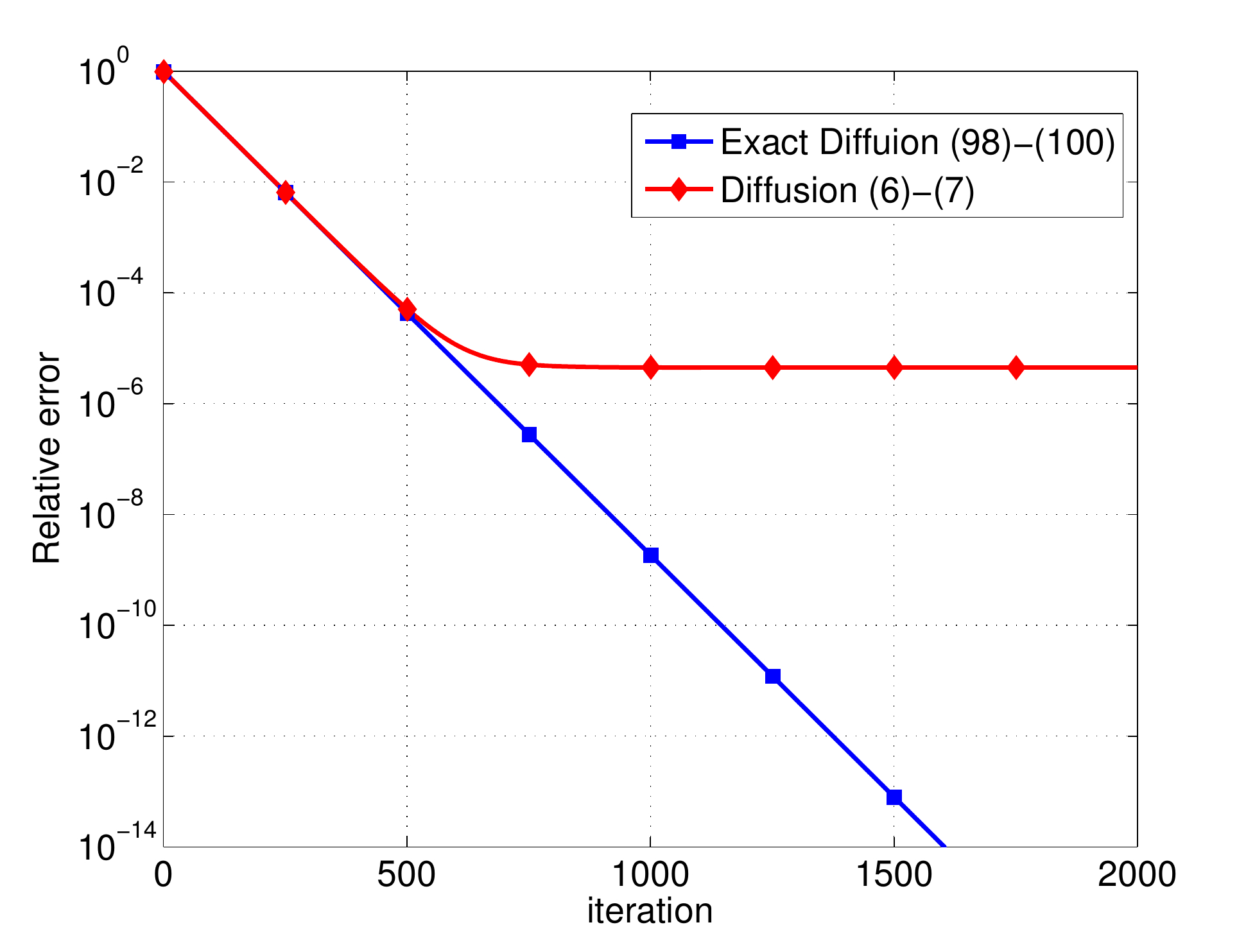}
		\vspace{-3mm}
		\caption{\footnotesize {Exact diffusion under the setting of Example 2 in Section \ref{sec-general-A}. Top: $\rho<1$ when $\mu<0.2$. Bottom: Convergence comparison between standard diffusion and exact diffusion when $\mu=0.001$. 
			}}
			\label{fig:general_A_converging}
		\end{figure}
		
\section{CONCLUDING REMARKS}
This work developed a diffusion optimization
strategy with guaranteed exact convergence for a broad
class of combination policies. The strategy  is applicable to non-symmetric
left-stochastic combination matrices, while many earlier
developments on exact consensus implementations are
limited to doubly-stochastic matrices or right-stochastic matrices; these latter matrices
impose stringent constraints on the network topology. Part II\cite{yuan2017exact2} of this work establishes analytically, and by means of examples and simulations, the superior convergence and stability properties of exact diffusion implementations.\\

\appendices

\section{Proof of \eqref{xcn23h}}
\label{app-unstable-A}	
		
\noindent {The characteristic polynomial of $\cF-\cG_{i-1}$ is given by
	\eq{\label{c-p}
		Q(\lambda) = (\lambda-1)D(\lambda),\quad \mbox{where} \quad D(\lambda) = \sum_{k=0}^{7} a_{k} \lambda^{k}
	} 
	and 
	\eq{
		a_7 &= 32, \quad
		a_6 = 384\mu - 128,\quad
		a_5 = 682\mu^2 \hspace{-0.8mm}-\hspace{-0.8mm} 1512\mu \hspace{-0.8mm}+\hspace{-0.8mm} 248, \nnb
		a_4 &= 429\mu^3 - 2458\mu^2 + 2712\mu - 288,\nnb
		a_3 &= 80\mu^4 - 1346\mu^3 + 3672\mu^2 - 2692\mu + 210,
	}
	\eq{
		a_2 &= - 240\mu^4 + 1649\mu^3 - 2904\mu^2 + 1593\mu - 98,\nnb
		a_1 &= 240\mu^4 - 976\mu^3 + 1260\mu^2 - 552\mu + 28, \nnb
		a_0 &= - 80\mu^4 + 244\mu^3 - 252\mu^2 + 92\mu - 4.
	}
	It is easy to observe from \eqref{c-p} that $\lambda=1$ is one eigenvalue of $\cF-\cG_{i-1}$. As mentioned in \eqref{23bb9} and its following paragraph, this eigenvalue $\lambda=1$ does not influence the convergence of recursion \eqref{xcn288sss} because of the initial conditions. It is the roots of $D(\lambda)$ that decide the convergence of the exact diffusion recursion \eqref{xcn288sss}. Now we will prove that there always exists some root that stays outside the unit-circle no matter what the step-size $\mu$ is. In other words, $D(\lambda)$ is not stable for any $\mu$.
	
	Since $D(\lambda)$ is a $7$-th order polynomial, its roots are not easy to calculate directly. Instead, we apply the Jury stability criterion \cite{jury1962simplified} to decide whether $D(\lambda)$ has roots outside  the unit-circle. First we construct the Jury table as shown in Fig. \ref{fig:jury-table}, where 
	\eq{
		b_k &= 
		\left|
		\begin{array}{cc}
			a_0 & a_{7-k} \\
			a_7 & a_k
		\end{array}
		\right| = a_0 a_k - a_7 a_{7-k},\ k = 0,\cdots, 6 \\
		c_k &= 
		\left|
		\begin{array}{cc}
			b_0 & b_{6-k} \\
			b_6 & b_k
		\end{array}
		\right| = b_0 b_k - b_6 b_{6-k},\ k = 0,\cdots, 5 \\
		\vdots \nnb
		f_k &= 
		\left|
		\begin{array}{cc}
			e_0 & e_{3-k} \\
			e_3 & e_k
		\end{array}
		\right| = e_0 e_k - e_3 e_{3-k},\ k = 0,\cdots, 2.
	}
	According to the Jury stability criterion, $D(\lambda)$ is stable (i.e., all roots of $D(\lambda)$ are within the unit-circle) if, and only if, the following conditions hold:
	\eq{
		& D(1) > 0,\quad (-1)^7D(-1)>0, \quad |a_0| < a_7, \quad |b_0| > |b_6|\nnb
		& |c_0| > |c_5|,\quad |d_0| > |d_4|, \quad |e_0| > |e_3|,\quad |f_0| > |f_2|.
	}
	If any one of the above conditions is violated, $D(\lambda)$ is not stable. Next we check each of the conditions:
	
	\vspace{1mm}
	\noindent (1) $D(1)>0$ is satisfied for any $\mu>0$ since 
	\eq{\label{xn3h9}
		D(1) = \sum_{k=0}^{7}a_k = 25 \mu > 0.
	}
	
	\vspace{1mm}
	\noindent (2) $(-1)^7D(-1)>0$. To guarantee this condition, we need to require that
	\eq{
		&\ (-1)^7 D(-1) \nnb
		=&\ 640\mu^4 - 4644\mu^3 + 11228\mu^2 - 9537\mu + 1036 > 0.
	}
	With the help of Matlab, we can verify that
	\eq{\label{b2v9}
	 (-1)^7 D(-1) > 0\ \mbox{ when } \mu<0.1265 \mbox{ or } \mu > 3.0410. 
	}
	
	\vspace{1mm}
	\noindent (3) $|a_0|<a_7$. To guarantee this condition, we need
	\eq{
		|\hspace{-1mm}-\hspace{-0.8mm}80\mu^4 + 244\mu^3 - 252\mu^2 + 92\mu - 4| < 32,
	}
	which is equivalent to requiring 
	\eq{
		- 0.1884 < \mu < 1.6323. \label{2bnn9}
	}
	With \eqref{xn3h9}, \eqref{b2v9} and \eqref{2bnn9}, we conclude that when 
	\eq{\label{2nbs9}
		0<\mu < 0.1265,
	}
	conditions (1), (2) and (3) will be satisfied simultaneously. Moreover, with the help of Matlab, we can also verify that the step-size range \eqref{2nbs9} will also meet conditions (4) $|b_0|>|b_6|$, (5) $|c_0|>|c_5|$ and (6) $|d_0|>|d_4|$. Now we check the last two conditions.
	\begin{figure}
		\centering
		\includegraphics[scale=0.4]{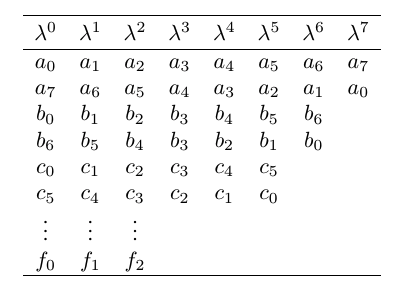}
		\caption{\footnotesize The Jury table for the $7$-th order system.}
		\label{fig:jury-table}
		\vspace{-5mm}
	\end{figure} 
	
	\vspace{1mm}
	\noindent (7) $|e_0|>|e_3|$. To guarantee this condition, the step-size $\mu$ is required to satisfy
	\eq{\label{23nd8}
		0.0438 < \mu < 0.1265.
	}
	
	\vspace{1mm}
	\noindent (8) $|f_0|>|f_2|$. To guarantee this condition, the step-size $\mu$ is required to satisfy
	\eq{\label{e99}
		0 < \mu < 0.0412.
	}
	Comparing \eqref{2nbs9}, \eqref{23nd8} and \eqref{e99}, it is observed that the intersection of these three ranges is empty, which implies that there does not exist a value for $\mu$ that makes all conditions (1)--(8) hold. Therefore, we conclude that $D(\lambda)$ is not stable for any step-size $\mu$. 
}

\bibliographystyle{IEEEbib}
\bibliography{reference}
\end{document}